\documentclass[a4paper,11pt,french,english]{amsart}
\usepackage[english]{babel}
\usepackage[latin1]{inputenc}
\usepackage{lmodern}
\usepackage[T1]{fontenc}
\usepackage{amsmath,amsthm,amssymb,amsfonts}
\usepackage{amscd}
\usepackage{pstricks}
\usepackage[a4paper,left=3.8cm,right=3.8cm,top=4cm,bottom=4cm]{geometry}
\usepackage{enumerate}
\usepackage{enumitem}
\usepackage[linktocpage=true]{hyperref}
\usepackage{url}

\makeatletter
\@namedef{subjclassname@2010}{%
  \textup{2010} Mathematics Subject Classification}
\makeatother

\newcommand{\pc}{\mathrm{Precone}}
\newcommand{\cone}{\mathrm{Cone}^{\omega}(G,\mathbf{s})}

\title{Locally compact lacunary hyperbolic groups}
\author{Adrien Le Boudec}
\address{Laboratoire de Mathématiques, bâtiment 425, Université Paris-Sud 11, 91405 Orsay, France}
\email{adrien.le-boudec@math.u-psud.fr}
\subjclass[2010]{20F65, 20F67, 20F69, 22D05, 20E08.}

\theoremstyle{plain}
\newtheorem{thm}{Theorem}[section]
\newtheorem{prop}[thm]{Proposition}
\newtheorem{cor}[thm]{Corollary}
\newtheorem{lem}[thm]{Lemma}

\newtheorem*{thm-intro}{Theorem}

\theoremstyle{definition}

\newtheorem{ex}[thm]{Example}
\newtheorem{rmq}[thm]{Remark}

\begin{document}

\maketitle

\begin{abstract}
We investigate the class of locally compact lacunary hyperbolic groups. We prove that if a locally compact compactly generated group $G$ admits one asymptotic cone that is a real tree and whose natural transitive isometric action is focal, then $G$ must be a focal hyperbolic group. As an application, we characterize connected Lie groups and linear algebraic groups over an ultrametric local field of characteristic zero having cut-points in one asymptotic cone. 

We prove several results for locally compact lacunary hyperbolic groups, and extend the characterization of finitely generated lacunary hyperbolic groups to the setting of locally compact groups. We moreover answer a question of Olshanskii, Osin and Sapir about subgroups of lacunary hyperbolic groups.
\end{abstract}

\setcounter{tocdepth}{1}

\tableofcontents

\section{Introduction}

\subsection{Locally compact hyperbolic groups}

%It is now commonly accepted that beyond the case of finitely generated groups, one can profitably study locally compact compactly generated groups by viewing them as metric spaces, and studying their large scale geometry.
If $G$ is a locally compact group and $S$ a compact generating subset, then $G$ can be equipped with the word metric associated to $S$. %Like in the discrete setting, different compact generating subsets give rise to bi-Lipschitz equivalent word metrics, so that the quasi-isometry class of the group $G$ is well defined.
A locally compact compactly generated group is hyperbolic if it admits some compact generating subset such that the associated word metric is Gromov-hyperbolic. By \cite[Corollary 2.6]{CCMT}, this is equivalent to asking that the group acts continuously, properly and cocompactly by isometries on some proper geodesic hyperbolic metric space. Examples of non-discrete hyperbolic groups include semisimple real Lie groups of rank one, or the full automorphism group of a semi-regular locally finite tree. We freely use the shorthand \textit{hyperbolic LC-group} for locally compact compactly generated hyperbolic group.

Finitely generated hyperbolic groups have received much attention over the last twenty-five years, and their study led to a rich and powerful theory. On the other hand, hyperbolic LC-groups have not been studied to the same extent, and this disparity leads to the natural problem of discussing the similarities and differences between the discrete and non-discrete setting. One positive result in this vein is the extension of Bowditch's topological characterization of discrete hyperbolic groups, as those finitely generated groups that act properly and cocompactly on the space of distinct triples of a compact metrizable space, to the setting of locally compact groups \cite{Car-Dree}. However it turns out that some hyperbolic LC-groups exhibit some completely opposite behavior to what happens for discrete hyperbolic groups: while a non-virtually cyclic finitely generated hyperbolic group always contains a non-abelian free group, some hyperbolic LC-groups are non-elementary hyperbolic and amenable. It follows from the work of Caprace, Cornulier, Monod and Tessera that those can be characterized in terms of the dynamics of the action of the group on its boundary, and that they coincide with the class of mapping tori of compacting automorphisms (see Theorem \ref{thm-ccmt-intro}).

\subsection{Lacunary hyperbolic groups}

The definition of asymptotic cones of a metric space makes sense for a locally compact compactly generated group $G$. Let $\mathbf{s} = (s_n)$ be a sequence of positive real numbers tending to infinity, and $\omega$ a non-principal ultrafilter. We denote by $\pc(G,\mathbf{s})$ the set of sequences $(g_n)$ in $G$ such that there exists some constant $C > 0$ so that the word length of $g_n$ is at most $C s_n$ for every $n \geq 1$; and equip it with the pseudo-metric $d_{\omega}((g_n),(h_n)) = \mathrm{lim}^{\omega} \,d_S(g_n,h_n)/s_n$. It inherits a group structure by component-wise multiplication, and the asymptotic cone $\cone$ of $G$ associated to the parameters $\mathbf{s}, \omega$ is the homogeneous space $\mathrm{Precone}(G,\mathbf{s}) \, / \, \mathrm{Sublin}^{\omega}(G,\mathbf{s})$, where $\mathrm{Sublin}^{\omega}(G,\mathbf{s})$ is the subgroup of sequences at distance $d_{\omega}$ zero from the identity. The group $\pc(G,\mathbf{s})$ can be viewed as a large picture of the group $G$, and the action of $\pc(G,\mathbf{s})$ on $\cone$ is inherited from the action of $G$ on itself. Asymptotic cones capture the large-scale geometry of the word metric on $G$. In some sense, the metric space $\cone$ reflects the properties of the group $G$ that are visible at scale $\mathbf{s}$. 

For example if $G$ is a hyperbolic LC-group, then all its asymptotic cones are real trees. Interestingly enough, thanks to a result of Gromov \cite{Gro-asy, Drutucone}, one can characterize hyperbolicity in terms of asymptotic cones: a locally compact compactly generated group is hyperbolic if and only if all its asymptotic cones are real trees. However there exist finitely generated non-hyperbolic groups with some asymptotic cone a real tree. The first example appeared in \cite{TV}, where small cancellation theory is used to construct a finitely generated group with one asymptotic cone a real tree, and one asymptotic cone that is not simply connected. The systematic study of the class of finitely generated groups with one asymptotic cone a real tree, called \textit{lacunary hyperbolic groups}, was then initiated in \cite{OOS}. Olshanskii, Osin and Sapir characterized finitely generated lacunary hyperbolic groups as direct limits of sequences of finitely generated hyperbolic groups satisfying some conditions on the hyperbolicity constants and injectivity radii \cite[Theorem 3.3]{OOS}. They also proved that the class of finitely generated lacunary hyperbolic groups contains examples of groups that are very far from being hyperbolic: a non-virtually cyclic lacunary hyperbolic group can have all its proper subgroups cyclic, can have an infinite center or can be elementary amenable.

Following \cite{OOS}, we call a locally compact compactly generated group \textit{lacunary hyperbolic} if one of its asymptotic cones is a real tree. For example if $X$ is a proper geodesic metric space with a cobounded isometric group action, and if $X$ has one asymptotic cone that is a real tree, then the full isometry group $G = \mathrm{Isom}(X)$ is a locally compact lacunary hyperbolic group, which has a priori no reason to be discrete.

By construction any asymptotic cone $\cone$ of a locally compact compactly generated group $G$ comes equipped with a natural isometric action of the group $\pc(G,\mathbf{s})$. So in particular if $G$ admits one asymptotic cone $\cone$ that is a real tree, then we have a transitive action by isometries of the group $\pc(G,\mathbf{s})$ on a real tree.  Recall that isometric group actions on real trees are classified as follows: if the the translation length is trivial then there is a fixed point or a fixed end, and otherwise either there is an invariant line, a unique fixed end or two hyperbolic isometries without common endpoint. It turns out that when $G$ is a hyperbolic LC-group, then for every choice of parameters $\mathbf{s}$ and $\omega$,  the asymptotic cone $\cone$ is a real tree, and the type of the action of $\pc(G,\mathbf{s})$ on $\cone$ is inherited from the type of the $G$-action on itself. Recall that a hyperbolic LC-group $G$ is called focal if its action on $\partial G$ has a unique fixed point. In particular when $G$ is a focal hyperbolic group, then for every scaling sequence $\mathbf{s}$ and non-principal ultrafilter $\omega$, the asymptotic cone $\cone$ is a real tree and the action of $\pc(G,\mathbf{s})$ on $\cone$ fixes a unique boundary point. This naturally leads to the question as to whether this phenomenon may appear when considering non-hyperbolic groups. Our first result shows that this is not the case. More precisely, we prove the following statement (see Theorem \ref{thm-focal}).

\begin{thm} \label{thm-intro-focal}
Let $G$ be a locally compact compactly generated group. Assume that $G$ admits \textnormal{one} asymptotic cone $\cone$ that is a real tree and such that the group $\pc(G,\mathbf{s})$ fixes a unique end of $\cone$. Then \mbox{$G = H \rtimes \mathbb{Z}$} or $H \rtimes \mathbb{R}$, where the element $1 \in \mathbb{Z}$ or $\mathbb{R}$ induces a compacting automorphism of $H$.
\end{thm}

Recall that an automorphism $\alpha \in \mathrm{Aut}(H)$ of a locally compact group $H$ is said to be compacting if there exists a compact subset $V \subset H$ such that for every $h \in H$, for $n$ large enough $\alpha^n(h) \in V$. In particular we recover and strengthen the implication $(i) \Rightarrow (iii)$ of the following theorem of Caprace, Cornulier, Monod and Tessera.

\begin{thm} \cite[Theorem 7.3]{CCMT} \label{thm-ccmt-intro}
If $G$ is a locally compact compactly generated group, then the following statements are equivalent:
\begin{enumerate}[label=(\roman*)]
\item $G$ is a focal hyperbolic group;
\item $G$ is amenable and non-elementary hyperbolic;
\item $G$ is a semidirect product $H \rtimes \mathbb{Z}$ or $H \rtimes \mathbb{R}$, where the element $1 \in \mathbb{Z}$ or $\mathbb{R}$ induces a compacting automorphism of the non-compact group $H$.
\end{enumerate}
\end{thm}

Our method is different from that of \cite{CCMT}: indeed the latter makes a crucial use of amenability, and the fact that quasi-characters on amenable groups are characters, while we only use geometric arguments at the level of the real tree arising as an asymptotic cone.

We point out that this strengthening of the result of \cite{CCMT} is definitely not the main application of Theorem \ref{thm-intro-focal} for our purpose, and that Theorem \ref{thm-intro-focal} is a crucial step in the proofs of both Theorem \ref{thm-intro-structure-lac-hyp} and Theorem \ref{thm-intro-cutpnt} below. 

We call a locally compact compactly generated group $G$ \textit{lacunary hyperbolic of general type} if it admits one asymptotic cone $\cone$ that is a real tree and such that the action of $\pc(G,\mathbf{s})$ has two hyperbolic isometries without common endpoint. Drutu and Sapir proved that any non-virtually cyclic finitely generated lacunary hyperbolic group is of general type (see the end of the proof of Theorem 6.12 in \cite{DS}). In the locally compact setting, it will follow from Theorem \ref{thm-intro-focal} that any lacunary hyperbolic group that is neither an elementary nor a focal hyperbolic group, is lacunary hyperbolic of general type (see Theorem \ref{thm-typegen}). 

%Recall that in any topological group $G$, the unit component $G^\circ$ is a closed normal subgroup such that the quotient group $G/G^\circ$ is totally disconnected.
It is often the case in topological group theory that a given problem can be reduced to the case of connected groups and totally disconnected groups, by using the fact that any topological group decomposes as an extension with connected kernel and totally disconnected quotient. For instance if one wants to study the large scale geometry of a given class of compactly generated groups (say that is stable by modding out by a compact normal subgroup and passing to a cocompact normal subgroup), then this can be reduced to the study of connected and totally disconnected groups as soon as the identity component of a group in this class is either compact or cocompact. It is worth pointing out that this process cannot be applied in generality for hyperbolic LC-groups, because it may happen that the unit component of a hyperbolic LC-group is neither compact nor cocompact. A typical example is $(\mathbb{Q}_p \times \mathbb{R}) \rtimes \mathbb{Z}$, where the automorphism of $\mathbb{Q}_p \times \mathbb{R}$ is the multiplication by $(p,p^{-1})$. However, apart from focal groups, it is true that the identity component of a hyperbolic LC-group is either compact or cocompact \cite[Proposition 5.10]{CCMT}. Here we will extend this result to the setting of lacunary hyperbolic groups in Theorem \ref{thm-composante-typegen}.

As a consequence, we will be able to deduce that a locally compact lacunary hyperbolic group is either hyperbolic or admits a compact open subgroup. Compactly generated groups with compact open subgroups are generally more tractable than compactly generated locally compact groups. For example they act geometrically on a locally finite connected graph thanks to a construction due to Abels recalled in Proposition \ref{Cay-Ab}. Most importantly for our purpose, the fact that any finitely generated group is a quotient of a finitely generated free group, admits a topological extension to the class of compactly generated groups with compact open subgroups (see Proposition \ref{prop-free}). This will allow us to extend the characterization of finitely generated lacunary hyperbolic groups of Olshanskii, Osin and Sapir to the locally compact setting (see Proposition \ref{prop-hyp-co} and Theorem \ref{thm-struct-lac}).

\begin{thm} \label{thm-intro-structure-lac-hyp}
Let $G$ be a compactly generated locally compact group. Then $G$ is lacunary hyperbolic if and only if \begin{enumerate}[label=(\alph*)]
\item either $G$ is hyperbolic; or 
\item there exists a hyperbolic LC-group $G_0$ acting geometrically on a locally finite tree, and an increasing sequence of discrete normal subgroups $N_n$ of $G_0$, whose discrete union $N$ is such that $G$ is isomorphic to $G_0 / N$; and if $S$ is a compact generating set of $G_0$ and \[ \rho_n = \min \{\left|g\right|_S \, : \, g \in N_{n+1} \! \setminus \! N_n \},\] then $G_0 / N_n$ is $\delta_n$-hyperbolic with $\delta_n =o(\rho_n)$.
\end{enumerate}
\end{thm}

\subsection{Subgroups of lacunary hyperbolic groups}

In \cite{OOS}, the authors initiated the study of subgroups of finitely generated lacunary hyperbolic groups. They proved for example that any finitely presented subgroup of a lacunary hyperbolic group is a subgroup of a hyperbolic group, or that a subgroup of bounded torsion of a lacunary hyperbolic group cannot have relative exponential growth. This prohibits Baumslag-Solitar groups, free Burnside groups with sufficiently large exponent or lamplighter groups from occurring as subgroups of a finitely generated lacunary hyperbolic group \cite[Corollary 3.21]{OOS}. These groups are examples of groups satisfying a law, and the authors ask whether it is possible that a non-virtually cyclic finitely generated group of relative exponential growth in a finitely generated lacunary hyperbolic group satisfies a law.

Let $G$ be a compactly generated group and $\mathbf{s}$ a scaling sequence. For every subgroup $H \leq G$, the set of $H$-valued sequences of $\pc(G,\mathbf{s})$ is a subgroup of $\pc(G,\mathbf{s})$, which will be denoted $\pc_G(H,\mathbf{s})$. In particular when $\cone$ is a real tree, we have an isometric action of the group $\pc_G(H,\mathbf{s})$ on the real tree $\cone$, and one might wonder what is the type of this action in terms of the subgroup $H$. In Section \ref{sec-subgroups} we carry out a careful study of the possible type of the action of $\pc_G(H,\mathbf{s})$ on $\cone$, which leads to the following result.

\begin{thm} \label{thm-law-exp-intro}
Let $G$ be a unimodular lacunary hyperbolic group. If $H \leq G$ is a compactly generated subgroup of relative exponential growth in $G$ not having $\mathbb{Z}$ as a discrete cocompact subgroup, then $H$ cannot satisfy a law.
\end{thm}

We point out that the unimodularity assumption is essential in Theorem \ref{thm-law-exp-intro}, even for lacunary hyperbolic groups of general type. When specified to the setting of discrete groups, Theorem \ref{thm-law-exp-intro} answers a question of Olshanskii, Osin and Sapir \cite[Question 7.2]{OOS}. This prohibits for example finitely generated solvable groups from appearing as subgroups of finitely generated lacunary hyperbolic groups (see Corollary \ref{cor-solv-lac}).

\subsection{Asymptotic cut-points}

Recall that a point $x \in X$ in a geodesic metric space is a cut-point if $X \setminus \{x \}$ is not connected. Finitely generated groups with cut-points in some asymptotic cone have been studied among others by Drutu-Osin-Sapir \cite{DS}, Drutu-Mozes-Sapir \cite{DMS} or Behrstock \cite{Beh}. The property of having cut-points in some asymptotic cone can be seen as a very weak form of hyperbolicity. In \cite{DS}, Drutu and Sapir proved that if a finitely generated non-virtually cyclic group $G$ satisfies a law, then $G$ does not have cut-points in any of its asymptotic cones. This result does not extend immediately to locally compact groups, as for instance the real affine group $\mathbb{R} \rtimes \mathbb{R}$ is non-elementary hyperbolic and solvable of class two. Nevertheless Theorem \ref{thm-intro-focal} will allow us to generalize the result of Drutu and Sapir to the locally compact setting, by proving that a locally compact group satisfying a law does not have cut-points in any of its asymptotic cones as soon as it is neither an elementary nor a focal hyperbolic group (see Theorem \ref{thm-law-cutpt}).

In cite \cite{DMS}, Drutu-Mozes and Sapir studied the existence of cut-points in asymptotic cones for lattices in higher rank semisimple Lie groups. They conjectured that such a lattice does not have cut-points in any of its asymptotic cones, and proved the conjecture in some classical examples. Here we consider the same problem not for lattices, but for connected Lie groups themselves. More precisely, we prove the following rigidity result for connected Lie groups and linear algebraic groups over the $p$-adics (see Corollary \ref{cor-connected-cut}).

\begin{thm} \label{thm-intro-cutpnt}
Let $G$ be either a connected Lie group or a compactly generated linear algebraic group over an ultrametric local field of characteristic zero. If $G$ has cut-points in one of its asymptotic cones, then $G$ is a hyperbolic group.
\end{thm}

\subsection{Structure of the paper}
In Section \ref{sec-preli} we recall the definition of asymptotic cones and some general facts about group actions on hyperbolic metric spaces and real trees, and give some background on locally compact groups. 

In Section \ref{sec-preli-cones} we establish some preliminary results about the dynamics of group actions on asymptotic cones.

Section \ref{sec-focal} contains the proof of Theorem \ref{thm-intro-focal}, which essentially consists in two steps. The first one is to prove that any group satisfying the hypotheses of Theorem \ref{thm-intro-focal} is a topological semidirect product $H \rtimes \mathbb{Z}$ or $H \rtimes \mathbb{R}$, and this will be achieved by considering the modular function and using geometric arguments at the level of the asymptotic cone that is a real tree. The second step in the proof is to show that the associated action is compacting, and this is again deduced from the focal dynamics at the level of the asymptotic cone. 

The end of Section \ref{sec-focal} is devoted to the application of Theorem \ref{thm-intro-focal} to the study of locally compact groups with asymptotic cut-points, and in particular it contains the proof of Theorem \ref{thm-intro-cutpnt}. 

In Section \ref{sec-gen-type} we derive from Theorem \ref{thm-intro-focal} that any locally compact lacunary hyperbolic group that is not a hyperbolic LC-group must be lacunary hyperbolic of general type (see Theorem \ref{thm-typegen}). This will allow us to obtain that any locally compact lacunary hyperbolic group is either a hyperbolic LC-group, or has compact open subgroups. This result is an essential step towards the proof of Theorem \ref{thm-intro-structure-lac-hyp}, which will be given at the end of Section \ref{sec-gen-type}.

Finally Section \ref{sec-subgroups} is devoted to the study of the structure of subgroups of lacunary hyperbolic groups. We point out that while the main concerns of the previous sections were non-discrete groups, all the results of Section \ref{sec-subgroups} encompass the case of discrete groups and are new even in this setting. In the first part of Section \ref{sec-subgroups} we obtain the interesting result that any quasi-isometrically embedded normal subgroup of a lacunary hyperbolic group is either compact or cocompact (see Proposition \ref{prop-normal-compact-cocompact}). The second part of Section \ref{sec-subgroups} contains the proof of Theorem \ref{thm-law-exp-intro}.

\subsection*{Acknowledgments}
Most of the problems discussed in this paper arose from discussions with Yves Cornulier. I am very grateful to him for his interest in this work, his many useful suggestions and his careful reading of the paper. I also thank Romain Tessera for his interest and useful discussions. 

\section{Preliminaries} \label{sec-preli}

%We begin by presenting definitions and background material.

\subsection{Asymptotic cones}
We start this section by recalling the definition of asymptotic cones. Let $\omega$ be a non-principal ultrafilter, i.e.\ a finitely additive probability measure on $\mathbb{N}$ taking values in $\left\{0,1\right\}$ and vanishing on singletons. A statement $\mathrm{P}(n)$ is said to hold \mbox{$\omega$-almost} surely if the set of integers $n$ such that $\mathrm{P}(n)$ holds has measure $1$. For any bounded function $f: \mathbb{N} \rightarrow \mathbb{R}$, there exists a unique real number $\ell$ such that for every $\varepsilon > 0$, we have $f(n) \in \left[\ell - \varepsilon, \ell + \varepsilon\right]$ \mbox{$\omega$-almost} surely. The number $\ell$ is called the limit of $f$ along $\omega$, and we denote $\ell = \mathrm{lim}^{\omega} f(n)$.

%Recall also that given a sequence $(a_n)$ in a topological space $A$, we say that $a \in A$ is an $\omega$-limit of $(a_n)$ if for every neighbourhood $U$ of $a$, $a_n \in U$ \mbox{$\omega$-almost} surely. $\omega$-limits always exist if $A$ is compact and are unique provided that $A$ is Hausdorff.

Consider a non-empty metric space $(X,d)$, a base point $e \in X$, and a scaling sequence $\mathbf{s} = (s_n)$, i.e.\ a sequence of positive real numbers tending to infinity. A sequence $(x_n)$ of elements of $X$ is said to be $\mathbf{s}$-linear if there exists a constant $C > 0$ so that $d(x_n,e) \leq C s_n$ for all $n \geq 1$. We denote by $\pc(X,d,\mathbf{s})$ the set of $\mathbf{s}$-linear sequences. If $\omega$ is a non-principal ultrafilter, the formula $d_{\omega}(x,y) = \mathrm{lim}^{\omega} \,d(x_n,y_n)/s_n$ makes $\pc(X,d,\mathbf{s})$ a pseudometric space, i.e.\ $d_{\omega}$ satisfies the triangle inequality, is symmetric and vanishes on the diagonal. The asymptotic cone $\mathrm{Cone}^{\omega}(X,d,\mathbf{s})$ of $(X,d)$ relative to the scaling sequence $\mathbf{s}$ and the non-principal ultrafilter $\omega$, is defined by identifying elements of $\mathrm{Precone}(X,d,\mathbf{s})$ at distance $d_{\omega}$ zero. More precisely, $\mathrm{Cone}^{\omega}(X,d,\mathbf{s})$ is the set of equivalence classes of $\mathbf{s}$-linear sequences, where $x,y \in \pc(X,d,\mathbf{s})$ are equivalent if $d_{\omega}(x,y) = 0$. We will denote by $(x_n)^{\omega}$ the class of the $\mathbf{s}$-linear sequence $(x_n)$.

If two metric spaces $X,Y$ are quasi-isometric, then their asymptotic cones corresponding to the same parameters $\mathbf{s}$ and $\omega$ are bi-Lipschitz homeomorphic.

%If $(X,d)$ is a homogeneous metric space, then its asymptotic cones do not depend on the choice of the observation points and will be denoted $\mathrm{Cone}^{\omega}(X,d,\mathbf{s})$. 

Now if $G$ is a locally compact compactly generated group, it can be viewed as a metric space when endowed with the word metric $d_S$ associated to some compact generating subset $S$. Since word metrics associated to different compact generating sets are bi-Lipschitz equivalent, $\pc(G,d_S,\mathbf{s})$ does not depend on the choice of $S$ and will be denoted by $\pc(G,\mathbf{s})$. It inherits a group structure by component-wise multiplication. For any non-principal ultrafilter $\omega$, the set of $\mathbf{s}$-linear sequences that are at distance $d_{\omega}$ zero from the constant sequence $(e)$ is a subgroup of $\pc(G,\mathbf{s})$, denoted by $\mathrm{Sublin}^{\omega}(G,\mathbf{s})$. The asymptotic cone $\mathrm{Cone}^{\omega}(G,d_S,\mathbf{s})$ is by definition the space of left cosets \[ \mathrm{Cone}^{\omega}(G,d_S,\mathbf{s}) = \mathrm{Precone}(G,\mathbf{s}) \, / \, \mathrm{Sublin}^{\omega}(G,\mathbf{s}),\] endowed with the metric $d_{\omega}((g_n)^{\omega},(h_n)^{\omega}) = \mathrm{lim}^{\omega} \,d_S(g_n,h_n)/s_n$. By construction the group $\pc(G,\mathbf{s})$ acts transitively by isometries on $\mathrm{Cone}^{\omega}(G,d_S,\mathbf{s})$. Note that as a set, $\mathrm{Cone}^{\omega}(G,d_S,\mathbf{s})$ does not depend on $S$. Moreover if $S_1$, $S_2$ are two compact generating sets, then the identity map is a bi-Lipschitz homeomorphism between $\mathrm{Cone}^{\omega}(G,d_{S_1},\mathbf{s})$ and $\mathrm{Cone}^{\omega}(G,d_{S_2},\mathbf{s})$. We will denote by $\cone$ the corresponding class of metric spaces up to bi-Lipschitz homeomorphism.

If $H$ is a subgroup of a locally compact compactly generated group $G$, then for every scaling sequence $\mathbf{s}$, we will denote by $\pc_G(H,\mathbf{s})$ the subgroup of $\pc(G,\mathbf{s})$ consisting of $H$-valued sequences. Remark that if $H$ is a normal subgroup of $G$ then $\pc_G(H,\mathbf{s})$ is normal in $\pc(G,\mathbf{s})$, and if $H$ satisfies a law then $\pc_G(H,\mathbf{s})$ satisfies the same law. These two simple observations will be used repeatedly throughout the paper.

\subsection{Isometric actions on hyperbolic spaces and real trees}

\subsubsection{Isometric actions on hyperbolic metric spaces and hyperbolic groups}

Let $X$ be a geodesic $\delta$-hyperbolic metric space, and $x \in X$ a base-point. Recall that it means that $X$ is a geodesic metric space such that any side of any geodesic triangle is contained in the $\delta$-neighbourhood of the union of the two other sides. We define the Gromov product relative to $x$ by the formula $2 (y,z)_x =d(y,x) + d(z,x) - d(y,z)$. A sequence $(y_n)$ of points in $X$ is called Cauchy-Gromov if \mbox{$(y_n,y_m)_x \rightarrow \infty$} as \mbox{$m,n \rightarrow \infty$}. The relation on the set of Cauchy-Gromov sequences defined by \mbox{$(y_n) \sim (z_n)$} if $(y_n,z_n)_x \rightarrow \infty$ as $n \rightarrow \infty$, is an equivalence relation, and the boundary $\partial X$ of the hyperbolic metric space $X$ is by definition the set of equivalence classes of Cauchy-Gromov sequences. 

Recall that if $\varphi$ is an isometry of $X$, then the quantity $d(\varphi^n x,x)/n$ always converges to some real number $l(\varphi) \geq 0$ as \mbox{$n \rightarrow \infty$}. When $l(\varphi) = 0$, the isometry $\varphi$ is called \textit{elliptic} if it has bounded orbits, and \textit{parabolic} otherwise. When $l(\varphi) > 0$, the isometry $\varphi$ is called \textit{hyperbolic}. The limit set of $\varphi$, also called the set of endpoints of $\varphi$, is the subset of $\partial X$ of Cauchy-Gromov sequences defined along an orbit of $\varphi$. It is empty if $\varphi$ is elliptic, a singleton if $\varphi$ is parabolic and has cardinality two if $\varphi$ is hyperbolic.

Now let $\Gamma$ be a group acting by isometries on $X$. Gromov's classification \cite{GroHyp}, which is summarized in Figure \ref{types}, says that exactly one of the following happens:
\begin{enumerate}
\item orbits are bounded, and the action of $\Gamma$ on $X$ is said to be \textit{bounded};
\item orbits are unbounded and $\Gamma$ does not contain any hyperbolic element, in which case the action is said to be \textit{horocyclic};
\item $\Gamma$ has a hyperbolic element and any two hyperbolic elements share the same endpoints. Such an action is termed \textit{lineal};
\item $\Gamma$ has a hyperbolic element, the action is not lineal and any two hyperbolic elements share an endpoint. In this situation we say that the action is \textit{focal};
\item there exist two hyperbolic elements not sharing any endpoint. Such an action is said to be of \textit{general type}.
\end{enumerate}

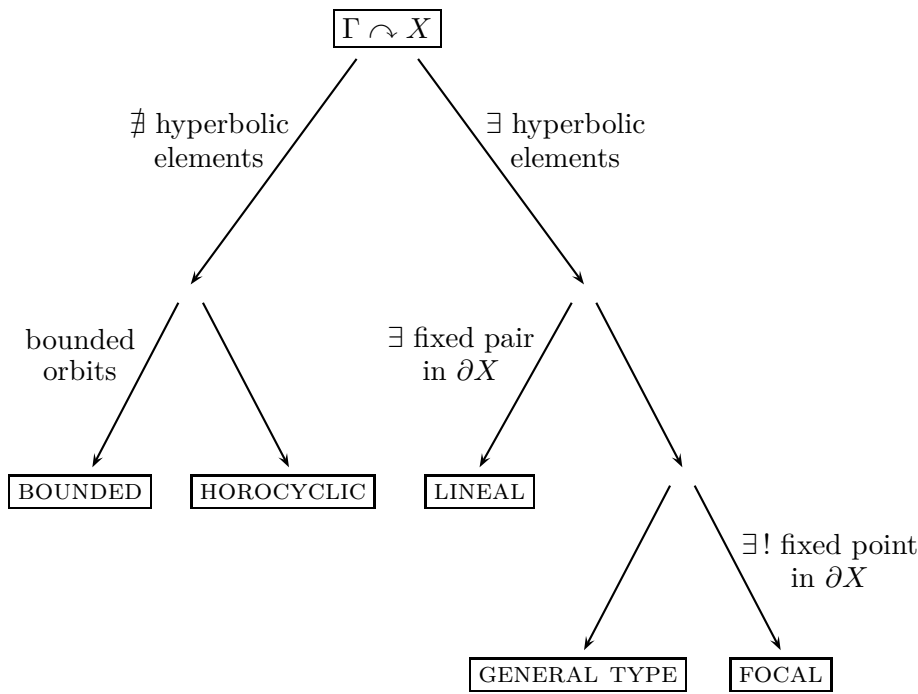
\begin{figure}
\psset{unit=23pt}
\begin{pspicture}(11,11.5)
\rput(5.5,11){\framebox{$\Gamma \curvearrowright X$}}
\psline[arrowsize=0.16]{->}(5,10.5)(2.3,6.8)
\rput(2.6,9.4){$\nexists$ hyperbolic}
\rput(2.6,8.9){elements} 
\psline[arrowsize=0.16]{->}(6,10.5)(8.7,6.8)
\rput(8.4,9.4){$\exists$ hyperbolic}
\rput(8.4,8.9){elements} 

\psline[arrowsize=0.16]{->}(2.1,6.5)(0.7,3.8)
\rput(0.5,5.9){bounded}
\rput(0.5,5.4){orbits}
\psline[arrowsize=0.16]{->}(2.5,6.5)(3.9,3.8)
\rput(0.5,3.4){\framebox{\textsc{bounded}}}
\rput(3.8,3.4){\framebox{\textsc{horocyclic}}}

\psline[arrowsize=0.16]{->}(8.5,6.5)(7,3.8)
\rput(6.7,5.9){$\exists$ fixed pair}
\rput(6.7,5.4){in $\partial X$}
\psline[arrowsize=0.16]{->}(8.9,6.5)(10.3,3.8)
\rput(7,3.4){\framebox{\textsc{lineal}}}

\psline[arrowsize=0.16]{->}(10.1,3.5)(8.7,0.8)
\psline[arrowsize=0.16]{->}(10.5,3.5)(11.9,0.8)
\rput(12.7,2.5){$\exists \, !$ fixed point}
\rput(12.7,2){in $\partial X$}
\rput(8.6,0.4){\framebox{\textsc{general type}}}
\rput(11.9,0.4){\framebox{\textsc{focal}}}
\end{pspicture}
\caption{\label{types}Types of actions on hyperbolic spaces}
\end{figure}

%Note that $\mathrm{Cay}(G,S)$ may very well be non-locally compact and the action of $G$ on $\mathrm{Cay}(G,S)$ non-continuous. Nevertheless, a locally compact group is hyperbolic if and only if it acts geometrically on a proper geodesic hyperbolic metric space \cite[Section 2]{CCMT}.

Now recall that a locally compact compactly generated group $G$ is called hyperbolic if its Cayley graph is hyperbolic for some (any) compact generating subset $S$. The type of $G$ is defined as the type of the action of $G$ on its Cayley graph. Since horocyclic isometric actions are always distorted (see for example Proposition 3.2 in \cite{CCMT}), hyperbolic LC-groups are never horocyclic. It is easily seen that a hyperbolic LC-group is bounded if and only if it is compact, and hyperbolic LC-groups that are lineal are exactly the locally compact compactly generated groups with two ends. These two types of hyperbolic LC-groups are usually gathered under the term of \textit{elementary} hyperbolic groups. 

When dealing with discrete groups, it is a classical result that a finitely generated non-elementary hyperbolic group is of general type. On the other hand, focal hyperbolic groups do exist in the realm of non-discrete locally compact groups. Examples include some connected Lie groups (e.g. $\mathbb{R}^{n-1} \rtimes \mathbb{R}$, $n \geq 2$, which admits a free and transitive isometric action on the $n$-dimensional hyperbolic space $\mathbb{H}^n$ fixing a boundary point), or the stabilizer of an end in the automorphism group of a semi-regular locally finite tree. Beyond the connected and totally disconnected cases, a simple example of a focal hyperbolic group is $(\mathbb{Q}_p \times \mathbb{R}) \rtimes \mathbb{Z}$, where the element $1 \in \mathbb{Z}$ acts by multiplication by $p$ on $\mathbb{Q}_p$ and by $p^{-1}$ on $\mathbb{R}$. %Observe for example that the real affine group $\mathbb{R} \rtimes \mathbb{R}$ has a free and transitive isometric action on the hyperbolic plane fixing a boundary point.
Caprace, Cornulier, Monod and Tessera characterized focal hyperbolic groups as those hyperbolic LC-groups that are non-elementary and amenable, and gave a precise description of the structure of these groups (see Theorem 7.3 in \cite{CCMT}).

\subsubsection{Actions on real trees}

We now recall some basic facts about real trees and isometric group actions on these. A metric space is a real tree if it is geodesic and $0$-hyperbolic, or equivalently if any two points are connected by a unique topological arc. If $T$ is a real tree, a non-empty subset $T'\subset T$ is called a subtree if it is connected, which is equivalent to saying that $T'$ is convex. We insist on the fact that by definition a subtree is necessarily non-empty. A point $x \in T$ is said to be a branching point if $T \setminus \{x\}$ has at least three connected components, and the branching cardinality of $x$ is the cardinality of the set of connected components of $T \setminus \{x\}$.

If $\varphi$ is an isometry of a real tree $T$, then the translation length of $\varphi$ is defined as \[ \left\| \varphi \right\| = \inf_{x \in T} d(\varphi x,x),\] and the characteristic set $\mathrm{Min}_{\varphi}$ of $\varphi$ is the set of points where the translation length is attained. The following proposition, a proof of which can be consulted in \cite{Cul-Mor}, shows that the dynamics of an individual isometry of a real tree is easily understood.

\begin{prop}
The characteristic set $\mathrm{Min}_{\varphi}$ is a closed subtree of $T$ which is invariant by $\varphi$. If $\left\|g\right\| = 0$ then $\varphi$ is elliptic and $\mathrm{Min}_{\varphi}$ is the set of fixed points of $\varphi$; and if $\left\| \varphi \right\| > 0$ then $\mathrm{Min}_{\varphi}$ is a line isometric to $\mathbb{R}$, called the axis of $\varphi$, along which $\varphi$ translates by $\left\| \varphi \right\|$.
\end{prop}

If $\Gamma$ is a group acting by isometries on a real tree $T$, an invariant subtree $T'$ is called minimal if it does not contain any proper invariant subtree. When this holds we also say that the action of $\Gamma$ on $T'$ is minimal, or that $\Gamma$ acts minimally on $T'$. Since a real tree is a hyperbolic metric space, the classification of isometric group actions on hyperbolic spaces recalled in the previous paragraph holds, and the five possible types of actions may occur for groups acting on real trees. However if the action of $\Gamma$ on $T$ is minimal, then this action cannot be bounded unless $T$ is reduced to a point, is never horocyclic, and is lineal if and only if $T$ is isometric to the real line.

The following lemma is standard, see Proposition 3.1 in \cite{Cul-Mor}.

\begin{lem}
Suppose that $\Gamma$ is a group acting on a real tree. If $\Gamma$ contains some hyperbolic element, then the union of the axes of the hyperbolic elements of $\Gamma$ is an invariant subtree contained in any other invariant subtree.
\end{lem}

A simple but useful consequence is the following result.

\begin{lem} \label{lem-norm-hyp}
Let $\Gamma$ be a group acting minimally on a real tree $T$, and let $\Lambda \triangleleft \Gamma$ be a normal subgroup containing some hyperbolic element. Then the action of $\Lambda$ on $T$ is minimal as well, and every point of $T$ lies on the axis of some hyperbolic element of $\Lambda$.
\end{lem}

\begin{proof}
Let $T'$ be the union of the axes of the hyperbolic elements of $\Lambda$, which is a minimal $\Lambda$-invariant subtree by the previous lemma. To prove the statement, it is enough to prove that $T'=T$. But this is clear because the condition that $\Lambda$ is a normal subgroup of $\Gamma$ implies that $T'$ is also a $\Gamma$-invariant subtree, and by minimality of the action of $\Gamma$ on $T$, one must have $T'=T$.
\end{proof}

%\begin{prop} \label{lac-elementary}
%Let $G$ be a locally compact lacunary hyperbolic group and let $(s_n)$, $\omega$ be such that $\mathrm{Cone}^{\omega}(G,(s_n))$ is a %real tree. The action of $\mathrm{Precone}(G,(s_n))$ on $\mathrm{Cone}^{\omega}(G,(s_n))$ is
%\begin{enumerate}
%\item bounded if and only if $G$ is compact;
%\item never horocyclic;
%\item lineal if and only if $G$ admits $\mathbb{Z}$ as a cocompact lattice.
%\end{enumerate}
%\end{prop}

\subsection{Locally compact groups}

We now aim to recall some structural results about locally compact compactly generated groups that will be needed later. As it is often the case, we will deal separately with connected and totally disconnected groups.

\subsubsection{Connected locally compact groups}

The material of this paragraph is classical. It is an illustration of how the solution of Hilbert's fifth problem can be used to derive results about connected locally compact groups from the study of connected Lie groups.

\begin{prop} \label{prop-radcomp-connec}
Every connected locally compact group has a unique maximal compact normal subgroup, called the compact radical, and the corresponding quotient is a connected Lie group.
\end{prop}

%\begin{proof}
%Let $G$ be a connected locally compact group. By Theorem 4.6 of \cite{MontZip}, there exists a compact normal subgroup $K$ of $G$ such that $G/K$ is a connected Lie group. Since on the one hand having a unique maximal compact normal subgroup is preserved by group extension with compact kernel, and on the other hand any connected Lie group has a unique maximal compact normal subgroup, the conclusion follows.%Since the statement holds for connected Lie groups, let $K_2$ be the pre-image of a maximal compact normal subgroup of $G/K_1$ under the map $\pi:G \twoheadrightarrow G/K_1$. It is easily checked that $K=K_1K_2$ is a compact normal subgroup of $G$, and it is maximal by maximality of $\pi(K_2)$ in $G/K_1$.
%\end{proof}

\begin{proof}
See Theorem 4.6 of \cite{MontZip}.
\end{proof}

If $G$ is a topological group, we denote by $G^\circ$ the connected component of the identity. It is a closed characteristic subgroup of $G$, and the quotient $G/G^\circ$, endowed with the quotient topology, is a totally disconnected group. 

\begin{cor} \label{cor-red-lie}
Every locally compact group $G$ has a compact subgroup $K$ that is characteristic and contained in $G^\circ$, such that the quotient $G^\circ/K$ is a connected Lie group without non-trivial compact normal subgroup.
\end{cor}

\begin{proof}
Take $K$ the compact radical of $G^\circ$. Being characteristic in the characteristic subgroup $G^\circ$, it is characteristic in $G$.
\end{proof}

The following result will be used in Section \ref{subsec-cut}.

\begin{cor} \label{cor-qi-law}
Every connected-by-compact locally compact group is quasi-isometric to a compactly generated solvable group.
\end{cor}

\begin{proof}
Clearly it is enough to prove the result for a connected locally compact group $G$. Modding out by the compact radical of $G$, we may assume by Proposition \ref{prop-radcomp-connec} that $G$ is a connected Lie group, and the result now follows from the classical fact that any connected Lie group has a (possibly non-connected) cocompact solvable Lie subgroup.
\end{proof}

\subsubsection{Locally compact groups with compact open subgroups}

Recall that if $G$ is a locally compact totally disconnected group, then according to van Dantzig's theorem, compact open subgroups of $G$ exist and form a basis of identity neighbourhoods. In this section we will deal with the slightly more general class of groups, namely the class of groups $G$ having compact open subgroups. Note that by van Dantzig's theorem, this is equivalent to saying that $G$ is a locally compact group with a compact identity component.

The following result, originally due to Abels, associates a connected locally finite graph to any compactly generated locally compact group with a compact open subgroup. 

\begin{prop} \label{Cay-Ab}
Let $G$ be a compactly generated locally compact group having a compact open subgroup. Then there exists a connected locally finite graph $X$ on which $G$ acts by automorphisms, transitively and with compact open stabilizers on the set of vertices.
\end{prop}

Recall that the construction consists in choosing a compact open subgroup $K$, and a compact generating subset $S$ of $G$ that is bi-invariant under the action of $K$. We take $G/K$ as vertex set for the graph $X$, and two different cosets $g_1K$ and $g_2K$ are adjacent if there exists $s \in S^{\pm 1}$ such that $g_2=g_1s$. The resulting graph is connected and locally finite. The action of $G$ on $X$ is vertex-transitive, and the stabilizer of the base-vertex is the compact open subgroup $K$. The graph $X$ is called the Cayley-Abels graph of $G$ associated to the compact open subgroup $K$ and compact generating subset $S$. %Note that the orbital map $G \rightarrow X$, $g \mapsto gK$, is onto and we have $d_X(g_1K,g_2K) \leq d_S(g_1,g_2) \leq d_X(g_1K,g_2K) + 1$ for every $g_1,g_2 \in G$. %In particular there exists a universal constant, not depending on $G$, $K$ or $S$, such that $(X,d_X)$ is $\delta$-hyperbolic if and only if $(G,d_G)$ is $C\delta$-hyperbolic.
In some sense, the following result is a topological analogue of the fact that any finitely generated group is a quotient of a finitely generated free group. The result is not new (see for example \cite[Proposition 8.A.15]{Cor-dlH}), but the proof we give here is different from the one in \cite{Cor-dlH}. 

\begin{prop} \label{prop-free}
Let $G$ be a compactly generated locally compact group having a compact open subgroup. Then there exists a compactly generated locally compact group $G_0$ acting on a locally finite tree, transitively and with compact open stabilizers on the set of vertices; and an open epimorphism $\pi: G_0 \twoheadrightarrow G$ with discrete kernel.
\end{prop}

\begin{proof}
Let $K$ be a compact open subgroup of $G$, and $S$ a $K$-bi-invariant compact symmetric generating subset of $G$ containing the identity. Note that this implies that $K \subset S$. We let $R_{K,S}$ be the set of words of the form $s_1 s_2 k^{-1}$, with $s_1,s_2 \in S$ and $k \in K$, when the relation $s_1s_2=k$ holds in the group $G$. We denote by $G_0$ the group defined by the abstract presentation $G_0 = \left\langle S \mid R_{K,S}\right\rangle$. Note that by construction, the group $G_0$ comes equipped with a natural morphism $\pi: G_0 \rightarrow G$, which is onto since $S$ is a generating subset of $G$.

We claim that $G_0$ admits a commensurated subgroup isomorphic to the subgroup $K$ of $G$. Indeed, let $K_0$ be the subgroup of $G_0$ generated by $K \subset S$ (here $K$ is seen as a subset of the abstract generating set $S$). To prove that $K_0$ is isomorphic to $K$, it is enough to prove that $K_0$ intersects trivially the kernel of $\pi$. But this is clear, because by construction all the relations in $G$ of the form $k_1 k_2 = k_3$ are already satisfied in $G_0$, so the map $\pi$ induces an isomorphism between $K_0$ and the subgroup $K$ of $G$. Now it remains to prove that $K_0$ is commensurated in $G_0$. Since by definition $S$ generates $G_0$, it is enough to prove that the subset $S$ commensurates $K_0$ in $G_0$. Being compact and open in $G$, the subgroup $K$ is commensurated in $G$. Therefore for every $s \in S$ there exists a finite index subgroup $K^{(s)} \leq K$ such that $s K^{(s)} s^{-1} \leq K$. This can be rephrased by saying that for every $k^{(s)} \in K^{(s)}$, there exists $k \in K$ such that $s k^{(s)} s^{-1} = k$. But now using twice the set or relators $R_{K,S}$, it is not hard to check that these relations hold in $G_0$ as well, which implies that the subgroup $K_0$ is commensurated in $G_0$. This finishes the proof of the claim.

Now if we equip $K_0$ with the pullback topology under the restriction of the map $\pi_{/K_0}: K_0 \stackrel{\sim}{\rightarrow} K$, we obtain a group topology on $G_0$ turning $K_0$ into a compact open subgroup \cite[Chapter 3]{Bourb-topo}. Note that by construction the epimorphism $\pi: G_0 \twoheadrightarrow G$ is open and has a discrete kernel (because the latter intersects trivially the open subgroup $K_0$).

To end the proof of the proposition, we need to construct a locally finite tree on which $G_0$ acts with the desired properties. Let us consider the Cayley-Abels graph $X$ of $G_0$ associated to $K_0$ and $S$. The action of $G_0$ on $X$ is transitive and with compact open stabilizers on the set of vertices, so the only thing that needs to be checked is that $X$ is a tree, i.e.\ $X$ does not have non-trivial loops. To every loop in $X$ can be associated a word $s_1 \cdots s_n$ so that the relation $s_1 \cdots s_n k = 1$ holds in $G_0$ for some $k \in K$. This means that in the free group over the set $S$, we have a decomposition of the form \[ s_1 \cdots s_n k = \prod_{i=1}^{N} w_i \left(s_{i,1} s_{i,2} k_i^{-1} \right) w_i^{-1},\] with $s_{i,1} s_{i,2} k_i^{-1} \in R_{K,S}$. Now remark that in $X$, any loop indexed by a word of the form $s_{i,1} s_{i,2} k_i^{-1} \in R_{K,S}$ is nothing but a simple backtrack, and it follows that we have a decomposition of our original loop as a sequence of backtracks. This implies that $X$ is a tree and finishes the proof.
\end{proof}

\section{Preliminary results on asymptotic cones} \label{sec-preli-cones}

This section gathers a few lemmas that will be used in the sequel. As we have seen earlier, any asymptotic cone of a locally compact compactly generated group comes equipped with a natural isometric group action. The next lemma describes to what extent this data varies for instance when modding out by a compact normal subgroup or passing to a cocompact normal subgroup. We point out that in the second statement, the assumption that $\pi(G)$ is normal in $Q$ is essential (think of $\mathbb{R} \rtimes \mathbb{R}$ inside $\mathrm{SL}_2(\mathbb{R})$).

\begin{lem} \label{lem-cone-sametype}
Consider a proper homomorphism with cocompact image \mbox{$\pi: G \rightarrow Q$} between locally compact compactly generated groups. Then for every scaling sequence $\mathbf{s}$ and non-principal ultrafilter $\omega$, the induced map at the level of asymptotic cones \mbox{$\tilde{\pi}: \cone \rightarrow \mathrm{Cone}^{\omega}(Q,\mathbf{s})$} is a bi-Lipschitz homeomorphism. 

If we assume in addition that $\cone$ (and hence $\mathrm{Cone}^{\omega}(Q,\mathbf{s})$) is a real tree and that $\pi(G)$ is normal in $Q$, then the actions of $\pc(G,\mathbf{s})$ on $\cone$ and of $\pc(Q,\mathbf{s})$ on $\mathrm{Cone}^{\omega}(Q,\mathbf{s})$ have the same type.
\end{lem}

\begin{proof}
Since the homomorphism $\pi$ has compact kernel and cocompact image, it is a quasi-isometry. Therefore  the map $\tilde{\pi}$ defined by $\tilde{\pi}((g_n)^{\omega}) = (\pi(g_n))^{\omega}$ is a bi-Lipschitz homeomorphism, which is equivariant under the actions of $\pc(G,\mathbf{s})$. 

It follows that $\mathrm{Cone}^{\omega}(Q,\mathbf{s})$ is a real tree if and only if $\cone$ is a real tree. When this is so and when $\pi(G)$ is supposed to be normal in $Q$, if $\pc(G,\mathbf{s})$ stabilizes some finite subset in the boundary of $\cone$, then the same holds for the group $\pc(Q,\mathbf{s})$. The converse implication being clear, the proof is complete.
\end{proof}

Recall that a metric space $(X,d)$ is coarsely connected if there exists a constant $c>0$ such that for any $x,y \in X$, there exists a sequence of points \mbox{$x = x_0, x_1, \ldots, x_n = y$} such that $d(x_i,x_{i+1}) \leq c$ for every $i=0,\ldots,n-1$. 

\begin{lem} \label{lem-cone-unb}
Let $(X,d)$ be a coarsely connected non-empty metric space. If $(X,d)$ is unbounded, then so are all its asymptotic cones. 
\end{lem}

\begin{proof}
Let $e \in X$ be a base point, $\mathbf{s}$ a scaling sequence and $\omega$ a non-principal ultrafilter. We prove the stronger statement that for every $\ell > 0$, there exists a point in $\mathrm{Cone}^{\omega}(X,d,\mathbf{s})$ at distance exactly $\ell$ from the point $(e)^{\omega}$.

Since $(X,d)$ is unbounded, for every $n \geq 1$ there is a point $x_n \in X$ at distance at least $\ell s_n$ from the base point $e$. Now by coarse connectedness, $x_n$ can be chosen to be at distance at most $\ell s_n + c$ from $e$, where $c > 0$ is the constant from the definition of coarse connectedness. By construction, the sequence $(x_n)$ defines a point $(x_n)^{\omega} \in \mathrm{Cone}^{\omega}(X,d,\mathbf{s})$ that is at distance $\ell$ to the point $(e)^{\omega}$. 
\end{proof}

\begin{lem} \label{lem-comp-fixpt}
Let $G$ be a compactly generated locally compact group, and $H$ a closed compactly generated subgroup of $G$. Then for any asymptotic cone of $G$, the following statements are equivalent:
\begin{enumerate}[label=(\roman*)]
\item $H$ is compact;
\item $\pc_G(H,\mathbf{s})$ fixes the point $(e)^{\omega} \in \cone$; 
\item $\pc_G(H,\mathbf{s})$ has a bounded orbit in $\cone$.
\end{enumerate}
\end{lem}

\begin{proof}
The implications $i) \Rightarrow ii) \Rightarrow iii)$ are trivial. Let us prove $iii) \Rightarrow i)$ by proving the contrapositive statement. 

Since $H$ is a closed compactly generated subgroup of $G$, the metric space $(H,d_G)$ is coarsely connected \cite[Proposition 4.B.8]{Cor-dlH}. So if $H$ is assumed not to be compact, it follows from Lemma \ref{lem-cone-unb} that none of the asymptotic cones of $(H,d_G)$ are bounded. But the asymptotic cone of $H$ with the induced metric from $G$ can be naturally identified with the orbit under $\mathrm{Precone}_G(H,\mathbf{s})$ of the point $(e)^{\omega} \in \cone$. So it follows that $\mathrm{Precone}_G(H,\mathbf{s})$ has one unbounded orbit, and since the action is isometric, every orbit must be unbounded.
\end{proof}

\begin{rmq} \label{rmq-failure}
We illustrate the failure of Lemma \ref{lem-comp-fixpt} when $H$ is not compactly generated. Let $G = \mathbb{F}_p( \! (t) \! ) \rtimes_t \mathbb{Z}$, where $\mathbb{F}_p( \! (t) \! )$ is the field of Laurent series over some finite field $\mathbb{F}_p$, and let $H$ be the subgroup generated by $(t^{- \alpha_n}, 0)$, $n\geq1$, where $\alpha_n = 2^{2^n}$. Then for any scaling sequence $\mathbf{s}$ such that \mbox{$\alpha_n <\!< s_n <\!< \alpha_{n+1}$} (take for example $s_n = 2^{3 \cdot 2^{n-1}}$) and for any non-principal ultrafilter $\omega$, the group $\pc_G(H,\mathbf{s})$ fixes the point $(e)^{\omega} \in \mathrm{Cone}^{\omega}(G,\mathbf{s})$, whereas $H$ is clearly not compact.
\end{rmq}

\begin{lem} 
Let $G$ be a compactly generated locally compact group, and let $N$ be a closed normal subgroup of $G$. Assume that $N$ is not cocompact in $G$. Then for every asymptotic cone $\cone$, there exists a bi-Lipschitz ray $\gamma: [0, +\infty[ \rightarrow \cone$ such that for every $t \geq 0$, \[ d_{\omega} \left( \gamma(t), \mathcal{C}_N \right) \geq c t \] for some constant $c > 0$, where $\mathcal{C}_N$ is the orbit of the point $(e)^{\omega}$ under $\pc_G(N,\mathbf{s})$.
\end{lem}

\begin{proof}
Since the group $G/N$ is non-compact, it has an infinite quasi-geodesic ray, that can be lifted to a quasi-geodesic ray $\rho : [0, +\infty[ \rightarrow G$ such that for every $t \geq 0$, $d_G(\rho(t),N) \geq c t$ for some constant $c$. Now we easily check that for every non-principal ultrafilter $\omega$ and scaling sequence $\mathbf{s}$, the $\omega$-limit of the quasi-geodesic ray $\rho$ in $\cone$ is a bi-Lipschitz ray satisfying the required property.
\end{proof}

\begin{cor} \label{lem-cocomp-cone}
Let $G$ be a compactly generated locally compact group, and let $N$ be a closed normal subgroup of $G$. If for some parameters $\omega, \mathbf{s}$ the action of $\pc_G(N,\mathbf{s})$ on $\cone$ is cobounded, then $N$ is cocompact in $G$.
\end{cor}

%Let $\Gamma$ be a group acting on a set $X$, and $\Lambda$ a normal subgroup of $\Gamma$. If $Y$ is a $\Lambda$-invariant subset of $X$, then $\gamma \cdot Y$ is also $\Lambda$-invariant for every $\gamma \in \Gamma$. In particular if $\Lambda$ admits a unique invariant subset $Y$ of a given cardinality $k$, then $Y$ is $\Gamma$-invariant. 

When $G$ is a compactly generated group with an asymptotic cone $\cone$ that is a real tree and such that the action of $\pc(G,\mathbf{s})$ is of general type, the five types of actions on real trees may happen for the action of $\pc_G(H,\mathbf{s})$ on $\cone$, where $H$ is a subgroup of $G$. However the situation is more restrictive under the additional assumption that $H$ is a normal subgroup.

\begin{lem} \label{lem-normal}
Let $G$ be a locally compact compactly generated group. Assume that $G$ admits an asymptotic cone $\cone$ that is a real tree and such that the action of $\pc(G,\mathbf{s})$ is of general type. Then for any normal subgroup $N$ of $G$, the action of $\pc_G(N,\mathbf{s})$ on $\cone$ is either bounded or of general type.
\end{lem}

\begin{proof}
If the group $\pc_G(N, \mathbf{s})$ preserves a finite subset in the boundary of $\cone$, then this finite subset is also preserved by $\pc(G,\mathbf{s})$ because $\pc_G(N, \mathbf{s})$ is normal in $\pc(G,\mathbf{s})$. By assumption this does not happen, so it follows that the action of $\pc_G(N, \mathbf{s})$ on $\cone$ is either bounded or of general type.
\end{proof}

We point out that it may happen that the group $\pc_G(N, \mathbf{s})$ fixes the point $(e)^{\omega} \in \cone$ even if $N$ is non-compact. Indeed, if $G$ is a non-virtually cyclic finitely generated lacunary hyperbolic group with an infinite center $Z$ (such groups have been constructed in \cite{OOS}), then the action of the abelian group $\pc_G(Z, \mathbf{s})$ cannot be of general type, and therefore must have a fixed point. 

%\begin{proof}
%Just apply Lemma \ref{lem-gpaction} with $X$ the boundary of $\cone$, $\Gamma = \pc(G,\mathbf{s})$ and $\Lambda = \pc_G(H,\mathbf{s})$.
%\end{proof}

Let $(X,d)$ be a non-empty metric space, and let $x_0 \in X$. Recall that an isometry $\varphi$ of $X$ is hyperbolic if the limit as \mbox{$n \rightarrow \infty$} of $d(\varphi^nx_0,x_0)/n$ is positive. If $G \leq \mathrm{Isom}(X)$ is a subgroup of the isometry group of $X$, we can endow $G$ with the pseudo-metric $d_{x_0}(g,h) = d(gx_0,hx_0)$. Note that for every scaling sequence $\mathbf{s}$ and non-principal ultrafilter $\omega$, the group $\pc(G,d_{x_0},\mathbf{s})$ admits a natural action on the asymptotic cone $\mathrm{Cone}^{\omega}(X,d,\mathbf{s})$. 

%\begin{lem} \label{lem-hyp-isom}
%If $G \leq \mathrm{Isom}(X)$ acts on $X$ with a hyperbolic isometry, then for any scaling sequence $\mathbf{s}$ and non-principal ultrafilter $\omega$, the group $\pc(G,d_{x_0},\mathbf{s})$ acts on $\mathrm{Cone}^{\omega}(X,d,\mathbf{s})$ with a hyperbolic isometry.
%\end{lem}

%\begin{proof}
%By assumption there exist $g \in G$ and $c > 0$ such that $d(g^n x_0,x_0) \geq c n$ for every $n \geq 1$. If we let $\mathbf{g} = (g^{s_n}) \in \pc(G,d_{x_0},\mathbf{s})$, then for every $k \geq 1$, \[ d_{\omega}(\mathbf{g}^k (x_0)^{\omega},(x_0)^{\omega}) = \mathrm{lim}^{\omega} \, \frac{d(g^{ks_n}x_0,x_0)}{s_n} \geq \mathrm{lim}^{\omega} \, \frac{c k s_n}{s_n} = c k, \] and a fortiori the limit as $k$ goes to infinity of $d_{\omega}(\mathbf{g}^k (x_0)^{\omega},(x_0)^{\omega}) / k$ is at least $c > 0$.
%\end{proof}

The following lemma says that if $X$ is a geodesic hyperbolic metric space and $G \leq \mathrm{Isom}(X)$, in many cases the type of the action of $\pc(G,d_{x_0},\mathbf{s})$ on $\mathrm{Cone}^{\omega}(X,d,\mathbf{s})$ is the same as the type of the action of $G$ on $X$. Note that both situations of statement (b) may happen (see Remark \ref{rmq-failure}).

\begin{lem} \label{lem-sametype-cone}
Let $X$ be a geodesic hyperbolic metric space, and $x_0 \in X$. If $G$ is a subgroup of the isometry group of $X$, then:
\begin{enumerate}[label=(\alph*)]
\item if the action of $G$ on $X$ is either bounded, lineal, focal or of general type, then for every asymptotic cone of $X$, the action of $\pc(G,d_{x_0},\mathbf{s})$ on $\mathrm{Cone}^{\omega}(X,d,\mathbf{s})$ has the same type;
\item if the action of $G$ on $X$ is horocyclic, then the action of $\pc(G,d_{x_0},\mathbf{s})$ on $\mathrm{Cone}^{\omega}(X,d,\mathbf{s})$ is either bounded or horocylic. %, and there exists a scaling sequence $\mathbf{s}$ such that this action is horocyclic for every non-principal ultrafilter $\omega$.
\end{enumerate} 
\end{lem}

\begin{proof}
We start by making the easy observation that if $g \in G$ is a hyperbolic element, then for every asymptotic cone $\mathrm{Cone}^{\omega}(X,d,\mathbf{s})$ of $X$, the element $(g^{s_n}) \in \pc(G,d_{x_0},\mathbf{s})$ is a hyperbolic isometry of $\mathrm{Cone}^{\omega}(X,d,\mathbf{s})$, and the axis of $(g^{s_n})$ is the asymptotic cone of any geodesic line in $X$ between the two endpoints of $g$. 

(a). The statement is obvious for bounded and lineal actions, and follows from the previous observation for actions of general type. Let us give the proof in the case when the action of $G$ on $X$ is focal. Let $\gamma \in G$ be a hyperbolic element. Since any two hyperbolic elements of $G$ share an endpoint, upon changing $\gamma$ into its inverse, we may assume that $(\gamma^k x_0) \in \partial X$ is the unique boundary point that is fixed by $G$. This implies (see \cite[Chap.7 Cor.3]{Ghys-dlH} or \cite[Part III.H Lemma 3.3]{BH}) that there exists some constant $c>0$ such that for every $g \in G$, we have $d(g \gamma^k x_0,\gamma^k x_0) \leq c d(g x_0,x_0)$ for every integer $k \geq 1$. It follows that for every element $(g_n) \in \pc(G,d_{x_0},\mathbf{s})$, there exists some constant $C > 0$ such that $d(g_n \gamma^{\left\lfloor ts_n\right\rfloor} x_0,\gamma^{\left\lfloor ts_n\right\rfloor}x_0) \leq C s_n$ for every $t \geq 0$ and $n \geq 1$. This implies that, if we let \mbox{$\xi: [0, + \infty [ \rightarrow \mathrm{Cone}^{\omega}(X,d,\mathbf{s})$} be the ray defined by $\xi(t) = (\gamma^{\left\lfloor ts_n\right\rfloor} x_0 )^{\omega}$, in the real tree $\mathrm{Cone}^{\omega}(X,d,\mathbf{s})$ the distance between $g \cdot \xi(t)$ and $\xi(t)$ is uniformly bounded, which means that the two rays $g \cdot \xi$ and $\xi$ represent the same end of $\mathrm{Cone}^{\omega}(X,d,\mathbf{s})$. Combined with the fact that $\pc(G,d_{x_0},\mathbf{s})$ contains hyperbolic elements not having the same endpoints (because $G$ already does), this implies that the action of $\pc(G,d_{x_0},\mathbf{s})$ on $\mathrm{Cone}^{\omega}(X,d,\mathbf{s})$ is focal.

(b). We assume that the action of $\pc(G,d_{x_0},\mathbf{s})$ on $\mathrm{Cone}^{\omega}(X,d,\mathbf{s})$ is not bounded, and we prove that it is horocylic. Let $(g_n x_0)^{\omega}$ be a point of $\mathrm{Cone}^{\omega}(X,d,\mathbf{s})$ such that $d_{\omega}((x_0)^{\omega},(g_n x_0)^{\omega}) = \ell > 0$, and let $\gamma_n$ be a geodesic in $X$ between $x_0$ and $g_n x_0$. Call $m_n$ the mid-point of $\gamma_n$. Recall that since the action of $G$ on $X$ is horocylic, for every $c > 0$ there exists some constant $c'$ such that the intersection in $X$ between any $c$-quasi-geodesic and any $G$-orbit lies in the union of two $c$-balls. This implies that \mbox{$\omega$-almost} surely the ball or radius $\ell s_n / 3$ around $m_n$ in $X$ does intersect the orbit $G x_0$. Therefore the mid-point of the unique geodesic in $\mathrm{Cone}^{\omega}(X,d,\mathbf{s})$ between $(x_0)^{\omega}$ and $(g_n x_0)^{\omega}$ is at distance at least $\ell / 3$ from any point in the $\pc(G,d_{x_0},\mathbf{s})$-orbit of $(x_0)^{\omega}$. In particular this proves that $\pc(G,d_{x_0},\mathbf{s})$ cannot preserve a geodesic line in $\mathrm{Cone}^{\omega}(X,d,\mathbf{s})$, and therefore does not have any hyperbolic isometry.
%We shall now prove the second part of the statement. Since the action of $G$ on $X$ is unbounded, there exists a sequence of elements $g_n \in G$ such that $d(g_n x_0, x_0)$ is increasing and tends to infinity. Let $s_n = d(g_n x_0, x_0)$, and let $\omega$ be any non-principal ultrafilter. By construction the element $(g_n)$ belongs to $\pc(G,d_{x_0},\mathbf{s})$ and sends the point $(x_0)^{\omega} \in \mathrm{Cone}^{\omega}(X,d,\mathbf{s})$ to a point at distance exactly one. Since any bounded action on a real tree has a fixed point, 
\end{proof}

\section{Focal lacunary hyperbolic groups} \label{sec-focal}

\subsection{Focal lacunary hyperbolic groups are focal groups}

This section is devoted to the proof of Theorem \ref{thm-intro-focal}. We call a locally compact compactly generated group \textit{focal lacunary hyperbolic} if it admits one asymptotic cone $\cone$ that is a real tree, and such that the action of $\pc(G,\mathbf{s})$ on $\cone$ is focal. According to Lemma \ref{lem-sametype-cone}, any focal hyperbolic group is a focal lacunary hyperbolic group. The rest of this section will be devoted to the proof of the following converse implication.

\begin{thm} \label{thm-focal}
Any focal lacunary hyperbolic group admits a topological semidirect product decomposition $H \rtimes \mathbb{Z}$ or $H \rtimes \mathbb{R}$, where the element $1 \in \mathbb{Z}$ or $\mathbb{R}$ acts on $H$ as a compacting automorphism. 
\end{thm}

Recall that an automorphism $\alpha \in \mathrm{Aut}(H)$ of a locally compact group $H$ is called compacting if there exists a compact subset $V \subset H$, called a pointwise vacuum set for $\alpha$, such that for every $h \in H$, there exists an integer $n_0 \geq 1$ such that $\alpha^n(h) \in V$ for every $n \geq n_0$. Note that $\alpha \in \mathrm{Aut}(H)$ is compacting if and only if some positive power of $\alpha$ is compacting.

The idea of the proof of Theorem \ref{thm-focal} is to deduce a contracting dynamics at the level of the group from a focal dynamics at the level of one asymptotic cone. The first step in the argument is to prove that a focal lacunary hyperbolic group is a topological semidirect product $H \rtimes \mathbb{Z}$ or $H \rtimes \mathbb{R}$. This will be achieved in Corollary \ref{cor-H-by-Z}.

Recall that if $G$ is a locally compact group endowed with the word metric associated to some compact generating subset, a cyclic subgroup $\left\langle g\right\rangle$ is said to be undistorted if the left multiplication by $g$ is a hyperbolic isometry of $(G,d_S)$, i.e.\ if the limit of $|g^n|_S /n$ is not zero. A sufficient condition for $\left\langle g\right\rangle$ to be undistorted is the existence of a continuous homomorphism $f: G \rightarrow \mathbb{Z}$ such that $f(g) \neq 0$.

The following result provides a criterion for a normal subgroup of a focal lacunary hyperbolic group to be cocompact. 

\begin{lem} \label{lem-focal-cocompact}
Let $G$ be a focal lacunary hyperbolic group, and $N$ a closed normal subgroup containing an undistorted element. Then $N$ is cocompact in $G$.
\end{lem}

\begin{proof}
We denote by $\mathcal{C} = \cone$ an asymptotic cone of $G$ that is a real tree and such that the action of $\pc(G,\mathbf{s})$ on $\mathcal{C}$ is focal. Let $\xi: [0, + \infty [ \rightarrow \mathcal{C}$ be the ray emanating from $(e)^{\omega}$ representing the end of $\mathcal{C}$ that is fixed by $\pc(G,\mathbf{s})$. Since the group $N$ contains an undistorted element, it follows that the group $\pc_G(N,\mathbf{s})$ acts on $\mathcal{C}$ with a hyperbolic element $h$, whose translation length will be denoted by $\ell$. Without loss of generality, we may assume that $(e)^{\omega}$ belongs to the axis of $h$. Indeed, if $(g_n)^{\omega}$ is a point on the axis of $h$ and if we denote by $g = (g_n)$, then $g^{-1} h g$ is hyperbolic and contains $(e)^{\omega}$ on its axis. Since $\pc_G(N,\mathbf{s})$ is normal in $\pc(G,\mathbf{s})$, the element $g^{-1} h g$ remains in $\pc_G(N,\mathbf{s})$, and the claim is proved. Now since the action of $\pc(G,\mathbf{s})$ on $\mathcal{C}$ is supposed to be focal, the axis of $h$ must contain the entire ray $\xi$.

Let us now prove that the action of $\pc_G(N,\mathbf{s})$ on $\mathcal{C}$ is cocompact. According to Corollary \ref{lem-cocomp-cone}, this finishes the proof of the proposition. Let $x$ be a point of $\mathcal{C}$. We will prove that the $\pc_G(N,\mathbf{s})$-orbit of $x$ in $\mathcal{C}$ intersects the segment joining $(e)^{\omega}$ and $\xi(\ell)$. According to Lemma \ref{lem-norm-hyp}, there exists some hyperbolic element $\gamma \in \pc_G(N,\mathbf{s})$ whose axis contains $x$. But since the action of $\pc(G,\mathbf{s})$ on $\mathcal{C}$ is focal, the axis of $\gamma$ intersects $\xi$ along an infinite ray, and by translating along the axis of $\gamma$, there exists some $n \in \mathbb{Z}$ so that $y = \gamma^n x$ belongs to $\xi$. But now since the axis of $h$ contains the ray $\xi$ and since $h$ translates along its axis by an amount of $\ell$, we can find $m \in \mathbb{Z}$ so that $h^m y$ remains in $\xi$ and is at distance at most $\ell$ from $(e)^{\omega}$.
\end{proof}

\begin{rmq} \label{rmq-lem-focal-cocomp}
Actually the same proof works with the only assumption that $N$ is a closed normal subgroup such that $\pc_G(N,\mathbf{s})$ acts on $\mathcal{C}$ with a hyperbolic element. This will be used in the proof of Proposition \ref{prop-contrac-cone}.
\end{rmq}

%\begin{cor}
%Let $G$ be a focal lacunary hyperbolic group such that $G^\circ$ is not of polynomial growth. Then $G^\circ$ is cocompact in $G$.
%\end{cor}

%\begin{proof}
%Again we may assume that $G^\circ$ is a Lie group. It follows from the work of Guivarc'h that connected Lie groups of polynomial growth are exactly those for which for every element $g$ of the group, the automorphism $\mathrm{Ad}_g$ of the Lie algebra has all its eigenvalues of modulus one. 

%Henceforth we assume that there is an element $g \in  G^\circ$ such that $\mathrm{Ad}_g$ has at least one eigenvalue of modulus $>1$. Such an element ***ARGUMENT*** will satisfy the property that the length of $g^n$ with respect to some compact generating set $S$ of $G$, grows linearly; that is there exists $c > 0$ such that for every $n \in \mathbb{Z}$, $c^{-1} |n| \leq |g^n|_S \leq c|n|$. 

%Cela donne un élément hyperbolique dans $\pc(G^\circ)$.
%\end{proof}

Recall that if $G$ is a locally compact group and $\mu$ a left-invariant Haar measure on $G$, for every $g \in G$ there exists a unique positive real number $\Delta_G(g)$ such that $\mu(Ug^{-1}) = \Delta_G(g) \mu(U)$ for every Borel subset $U$. The function $\Delta_G: G \rightarrow \mathbb{R}^{*}_+$ is called the modular function of $G$, and is a continuous group homomorphism.

The following proposition, which is a crucial step in the argument, consists in obtaining an estimate on the modular function of a focal lacunary hyperbolic group. In the proof, we take advantage of an idea appearing in the end of the proof of Theorem 6.12 in \cite{DS}.

\begin{lem} \label{lem-focal-unim}
Let $G$ be focal lacunary hyperbolic group, and $\mathcal{C} = \cone$ an asymptotic cone of $G$ that is a real tree and such that the action of $\pc(G,\mathbf{s})$ is focal. Let $\xi: [0, + \infty [ \rightarrow \mathcal{C}$, $\ell \mapsto (\xi_n(\ell))^{\omega}$, be the geodesic ray emanating from $(e)^{\omega}$ representing the end of $\mathcal{C}$ that is fixed by $\pc(G,\mathbf{s})$. Then there exist some constants $c>0$, $\rho > 1$, such that for every $\ell \geq 0$, we have \[ c \rho^{\ell s_n} \leq \Delta_G(\xi_n(\ell)) \] \mbox{$\omega$-almost} surely.
\end{lem}

%\begin{lem} \label{lem-focal-unim}
%Any focal lacunary hyperbolic group is non-unimodular.
%\end{lem}

\begin{proof}
%Let $G$ be a focal lacunary hyperbolic group, and $\mathcal{C} = \cone$ an asymptotic cone of $G$ that is a real tree such that the action of $\pc(G,\mathbf{s})$ on $\mathcal{C}$ is focal. We denote by $\xi: [0, + \infty [ \rightarrow \mathcal{C}$, $\ell \mapsto (\xi_n(\ell))^{\omega}$, the ray emanating from $(e)^{\omega}$ representing the end of $\mathcal{C}$ which is fixed by $\pc(G,\mathbf{s})$.

First note that for every $\ell \geq 0$, the element $(\xi_n(\ell)) \in \pc(G,\mathbf{s})$ sends the point $(e)^{\omega}$ to $\xi(\ell)$ by definition. But since the action of $\pc(G,\mathbf{s})$ on $\mathcal{C}$ is supposed to be focal, the image of the geodesic ray $\xi$ by $(\xi_n(\ell))$ eventually coincides with $\xi$. It follows that $(\xi_n(\ell)) \cdot \xi$ is exactly the infinite subray of $\xi$ emanating from $\xi(\ell)$. In particular for every $k \geq 1$, $(\xi_n(\ell)) \cdot \xi(k \ell) = \xi((k+1) \ell)$, and by a straightforward induction we obtain $\xi(k \ell) = ( \xi_n(\ell)^k )^{\omega}$.

Let $S$ be a compact generating subset of $G$, and denote by $B_S(r)$ the closed ball of radius $r \geq 0$ around the identity with respect to the word metric associated to $S$. Let $\ell \geq 0$, and $(g_n) \in \pc(G,\mathbf{s})$ such that $|g_n| \leq \ell s_n$ for every $n \geq 1$. The image of the point $(e)^{\omega}$ under such an element $(g_n)$ is at distance at most $\ell$ from $(e)^{\omega}$. The action being focal, it follows from this observation that the two rays $(g_n) \cdot \xi$ and $\xi$ intersect along an infinite subray of $\xi$ containing the point $\xi(\ell)$.

Now let us assume for a moment that the element $(g_n)$ is either elliptic or has the fixed end of $\mathcal{C}$ for attractive endpoint. Since the translation length of $(g_n)$ is at most $\ell$, it follows from the above observation that $(g_n) \cdot \xi(\ell)$ is at distance at most $\ell / 2$ from either $\xi(\ell)$ or $\xi(2 \ell)$. This implies that \mbox{$\omega$-almost} surely \[ d_S \left( g_n \xi_n(\ell), \left\{\xi_n(\ell), \xi_n(\ell)^2 \right\} \right) \leq \frac{2}{3} \ell s_n, \] where $d_S(g, \left\{h,k\right\})$ is by definition the minimum between $d_S(g,h)$ and $d_S(g,k)$. This inequality can be reformulated by saying that \mbox{$\omega$-almost} surely \[ g_n \in  \xi_n(\ell) \cdot B_S \left( 2 \ell s_n / 3 \right) \cdot \xi_n(\ell)^{-1} \, \cup \, \xi_n(\ell)^2 \cdot B_S \left( 2 \ell s_n / 3 \right) \cdot \xi_n(\ell)^{-1}. \] Now if the element $(g_n)$ is a hyperbolic isometry having the fixed end of $\mathcal{C}$ for repulsive endpoint, then we can apply the previous argument to $(g_n^{-1})$.

So we have proved that for every $\ell \geq 0$, \mbox{$\omega$-almost} surely the ball of radius $\ell s_n$ around the identity in $G$ lies inside \[ \xi_n(\ell) \cdot B_S \left( 2 \ell s_n / 3 \right) \cdot \xi_n(\ell)^{-1} \, \cup \, \xi_n(\ell)^2 \cdot B_S \left( 2 \ell s_n / 3 \right) \cdot \xi_n(\ell)^{-1} \, \cup \, \xi_n(\ell) \cdot B_S \left( 2 \ell s_n / 3 \right) \cdot \xi_n(\ell)^{-2}. \] Now if we let $\mu$ be a left-invariant Haar measure on $G$, then for every $\ell \geq 0$, \mbox{$\omega$-almost} surely \[ \mu \left( B_S(\ell s_n) \right) \leq 2 \mu \left( B_S \left( 2 \ell s_n / 3 \right) \cdot \xi_n(\ell)^{-1} \right) + \mu \left( B_S \left( 2 \ell s_n / 3 \right) \cdot \xi_n(\ell)^{-2} \right). \] Dividing by $\mu \left( B_S \left( 2 \ell s_n / 3 \right) \right)$, we obtain \[ \frac{\mu \left( B_S \left(\ell s_n \right) \right)}{\mu \left( B_S \left( 2 \ell s_n / 3 \right) \right)} \leq 2 \Delta_G (\xi_n(\ell)) + \Delta_G (\xi_n(\ell))^2 \leq 3 \Delta_G (\xi_n(\ell))^2. \]
We claim that the Haar-measure $\mu$ is not right-invariant. Let us argue by contradiction and assume that $\mu$ is right-invariant, which implies that the right-hand side of the last inequality is constant equal to $3$. Then for every $\ell \geq 0$, \mbox{$\omega$-almost} surely $\mu \left( B_S(\ell s_n) \right) \leq  3 \mu \left( B_S \left( 2 \ell s_n / 3 \right) \right)$. %Such an inequality on the growth function of $G$ implies that the asymptotic cone $\mathcal{C}$ is proper \cite[p.68]{GroPoly}, which is a contradiction because by assumption $\mathcal{C}$ is a real tree of continuum branching at every point. So $\mu$ cannot be right-invariant, i.e.\ $G$ is non-unimodular.
Now since every point of the real tree $\mathcal{C}$ is a branching point, spheres of any given radius in $\mathcal{C}$ are infinite, and it is not hard to see that this establishes a contradiction with the above inequality on the growth function of $G$. So $\mu$ cannot be right-invariant, i.e.\ $G$ is non-unimodular. In particular the group $G$ has exponential growth, and we easily deduce that the left-hand side of the above inequality is at least $c_1 \alpha^{\ell s_n}$ for some constants $c_1 > 0$, $\alpha >1$, and the conclusion follows with $c = \sqrt{c_1 / 3}$ and $\rho = \sqrt{\alpha}$.
\end{proof}

%We omit the proof of the following easy lemma.

%\begin{lem} \label{lem-gp-topo}
%Let $G$ be a metrizable topological group, $\Gamma$ a discrete subgroup and $K$ a compact subset of $G$. Then $\Gamma K$ is a closed subset of $G$.
%\end{lem}

\begin{cor} \label{cor-H-by-Z}
If $G$ is a focal lacunary hyperbolic group, then $G$ admits a topological semidirect product decomposition $H \rtimes \mathbb{Z}$ or $H \rtimes \mathbb{R}$, where $H$ is the kernel of the modular function of $G$.
\end{cor}

\begin{proof}
Let $H$ be the kernel of the modular function $\Delta_G: G \rightarrow \mathbb{R}^{*}_+$. Assume that we have proved that the image of $\Delta_G$ is a closed non-trivial subgroup of $\mathbb{R}^{*}_+$. Then the image of $\Delta_G$ is either discrete and infinite cyclic, or topologically isomorphic to $\mathbb{R}$. In the first case we easily have $G = H \rtimes \mathbb{Z}$, and in the other case we use the fact that any quotient homomorphism from a locally compact group to the group $\mathbb{R}$ is split, and deduce that $G = H \rtimes \mathbb{R}$.

So we should prove that the image of $\Delta_G$ is closed and non-trivial. According to Lemma \ref{lem-focal-unim}, we can choose some $\xi_n(\ell) = \gamma \in G$ such that $\Delta_G(\gamma) > 1$. Let us consider the subgroup $N = H \rtimes \left\langle \gamma\right\rangle$ of $G$ generated by $H$ and $\gamma$. Since $H$ contains the derived subgroup of $G$, the subgroup $N$ is normal in $G$, and being the preimage by $\Delta_G$ of the discrete subgroup of $\mathbb{R}^{*}_+$ generated by $\Delta_G(\gamma)$, $N$ is a closed subgroup. 

Since $\Delta_G(\gamma) \neq 1$, the cyclic subgroup generated by $\gamma$ is undistorted in $G$, and therefore we are in position to apply Lemma \ref{lem-focal-cocompact}, which implies that $N$ is a cocompact subgroup of $G$. Hence $\Delta_G$ induces a homomorphism from the compact group $G/N$ to $\mathbb{R}^{*}_+/\Delta_G(N)$, which necessarily has a closed image. Being the preimage in $\mathbb{R}^{*}_+$ of this closed subgroup, the image of $\Delta_G$ is closed.
\end{proof}

So we have proved that any focal lacunary hyperbolic group is either of the form $H \rtimes \mathbb{Z}$ or $H \rtimes \mathbb{R}$. We must now prove that the associated action is compacting. The first step towards this result is the following proposition, which says that a focal lacunary hyperbolic group satisfies in some sense a weak local contracting property.

If $H$ is a subgroup of a compactly generated group $G$, we denote by $B_{G,H}(r)$ the closed ball of radius $r \geq 0$ in $H$ around the identity, where $H$ is endowed with the induced metric from $G$.

\begin{prop} \label{prop-contrac-cone}
Let $G$ be a focal lacunary hyperbolic group. Assume that $G$ admits a topological semidirect product decomposition $G = H \rtimes \left\langle t_0 \right\rangle$. Then there exist $t \in \{t_0, t_0^{-1}\}$ and infinitely many $N \geq 1$ such that \[
t^{N} \cdot B_{G,H}(2N) \cdot t^{-N} \subset B_{G,H}(N).\]
\end{prop}

\begin{proof}
Let us denote by $\mathcal{C} = \cone$ an asymptotic cone of $G$ that is a real tree and such that the action of $\pc(G,\mathbf{s})$ on $\mathcal{C}$ is focal. Observe that the element $(t_0^{s_n}) \in \pc(G,\mathbf{s})$ is hyperbolic, and its axis is the image of the map $\mathbb{R} \rightarrow \cone$,  $x \mapsto (t_0^{- \left\lfloor  x s_n \right\rfloor })^{\omega}$. One of the two ends of this axis must be the end of $\mathcal{C}$ that is fixed by $\pc(G,\mathbf{s})$, so there is $t \in \{t_0, t_0^{-1} \}$ such that the ray emanating from $(e)^{\omega}$ representing the fixed end of $\mathcal{C}$ is the image of $\xi: [0, + \infty [ \rightarrow \cone$,  $x \mapsto (t^{- \left\lfloor  x s_n \right\rfloor })^{\omega}$.

We claim that $\pc_G(H,\mathbf{s})$ cannot fix a point in $\mathcal{C}$. Indeed, if the set of fixed points of $\pc_G(H,\mathbf{s})$ is not empty, then it is a subtree of $\mathcal{C}$ that is invariant by $\pc(G,\mathbf{s})$ since $H \triangleleft G$. But $\pc(G,\mathbf{s})$ acts transitively on $\mathcal{C}$, so we deduce that the set of fixed points of $\pc_G(H,\mathbf{s})$ is the entire $\mathcal{C}$. It follows that the action of $\pc_G(H,\mathbf{s})$ on $\mathcal{C}$ is trivial, and this implies that the asymptotic cone $\mathcal{C}$ is a line, which contradicts the fact that the action of $\pc(G,\mathbf{s})$ on $\mathcal{C}$ is focal. On the other hand, if $\pc_G(H,\mathbf{s})$ contains some hyperbolic isometry, then according to Remark \ref{rmq-lem-focal-cocomp} the conclusion of Lemma \ref{lem-focal-cocompact} holds and the subgroup $H$ is cocompact in $G$, which is a contradiction. So the action of $\pc_G(H,\mathbf{s})$ on $\mathcal{C}$ must be horocyclic. It follows that if $(h_n)$ is a sequence in $H$ such that $|h_n|_S \leq 2 \ell s_n$ for every $n \geq 1$ (which implies that the distance in $\cone$ between $(e)^{\omega}$ and $(h_n)^{\omega}$ is at most $2 \ell$), then the element $(h_n)$ fixes $\xi([\ell, + \infty [)$. In particular if $|h_n|_S \leq 2 s_n$ for every $n \geq 1$, then $(h_n)$ fixes the point $\xi(1) = (t^{-s_n})^{\omega}$, and we have \[ \mathrm{lim}^{\omega} \frac{d(h_n t^{-s_n},t^{-s_n})}{s_n} = \mathrm{lim}^{\omega} \frac{|t^{s_n} h_n t^{-s_n}|_S}{s_n} = 0. \] So for every $h_n \in  B_{G,H}(2s_n)$, \mbox{$\omega$-almost} surely we have $|t^{s_n} h_n t^{-s_n}|_S \leq s_n$, which is equivalent to saying that \mbox{$\omega$-almost} surely $t^{s_n} \cdot B_{G,H}(2s_n) \cdot t^{-s_n} \subset B_{G,H}(s_n)$. 
%We argue by contradiction. Henceforth we assume that there is $\lambda > 1 $ so that for every $n \geq 1$, there exists an element $h_n \in H$ of length at most $2s_n$ such that \[ |\alpha^{-s_n} h_n |_S = d_S(h_n, \alpha^{s_n}) \geq \lambda s_n . \] This inequality implies that the point $(h_n)^{\omega}$ is at distance at least $\lambda$ from $[\alpha^{s_n}]$ in $\mathcal{C}$. But this cannot happen because any element of the horosphere $\mathrm{Cone}(H)$ which is at distance at most $2$ from the point $(e)^{\omega}$ is at distance exactly $1$ from the point $[\alpha^{s_n}]$. Contradiction.
\end{proof}

%\begin{lem} \label{lem-comp-inc}
%Let $G = H \rtimes \left\langle \alpha \right\rangle$ be a focal lacunary hyperbolic group, where $H$ is the kernel of $\Delta_G$ and $\Delta_G(\alpha) > 1$. There exists some compact subset $K \subset H$ such that $\alpha^{-1} K \alpha \subset K$ and \[H = \bigcup_{n \geq 0}^{\nearrow} \alpha^{n} \left\langle K \right\rangle \alpha^{-n}. \]
%In particular for every $n_0 \geq 1$, $\left\langle K, \alpha^{n_0} \right\rangle = H \rtimes \left\langle \alpha^{n_0} \right\rangle$.
%\end{lem}

%\begin{defi}
%An automorphism $\alpha$ of a group $H$ is said to be strictly confining into a subset $A \subset H$ if:
%\begin{itemize}
%\item $\alpha(A)$ is strictly contained in $A$;
%\item $H$ is equal to the increasing union of the $\alpha^{-n}(A)$, $n \geq 0$;
%\item $\alpha^{n_0}(A^2)$ is contained in $A$ for some integer $n_0 \geq 1$.
%\end{itemize}
%\end{defi}

\begin{cor} \label{prop-sbgfi}
Let $G$ be a focal lacunary hyperbolic group with a topological semidirect product decomposition $G = H \rtimes \left\langle t_0 \right\rangle$. Then there exist $t \in \{t_0, t_0^{-1} \}$, an integer $n_0 \geq 1$ and a compact symmetric subset $K \subset H$ containing the identity such that:
\begin{enumerate}
\item [i)] $\left\langle K, t^{n_0} \right\rangle = H \rtimes \left\langle t^{n_0} \right\rangle$;
\item [ii)] $t^{n_0} \cdot K^2 \cdot t^{-n_0} \subset K$.
\end{enumerate}
\end{cor}

\begin{proof}
Let $t$ coming from Proposition \ref{prop-contrac-cone}, and $N_0 \geq 1$ an integer such that $B_{G,H}(N_0)$ together with $t$ generate the group $G$. According to Proposition \ref{prop-contrac-cone}, there exists $N_1 \geq N_0$ such that \[ t^{N_1} \cdot B_{G,H}(2N_1) \cdot t^{-N_1} \subset B_{G,H}(N_1).\] If we set \[ K_1 = \bigcup_{i=0}^{N_1-1} t^{i} \cdot B_{G,H}(N_1) \cdot t^{-i}, \] then $K_1$ is a compact subset of $H$ and by construction conjugating by $t$ sends $K_1$ into itself because \[ t \cdot K_1 \cdot t^{-1} \subset K_1 \cup t^{N_1} \cdot B_{G,H}(N_1) \cdot t^{-N_1} \subset K_1.\] In particular the sequence of compact subsets $(t^{-n} \cdot K_1 \cdot t^{n})_{n \geq 0}$, is increasing. A fortiori the same holds for the sequence of subgroups $(t^{-n} \cdot \left\langle K_1 \right\rangle \cdot t^{n})_{n \geq 0}$, and it follows that the subgroup they generate is nothing but their union. But now by assumption $K_1$ and $t$ generate $G$, so this increasing union of subgroups is the entire subgroup $H$. This observation implies in particular that for every $n_0 \geq 1$, the subgroup generated by $K_1$ and $t^{n_0}$ is equal to $H \rtimes \left\langle t^{n_0} \right\rangle$.

Now we let $n_0$ be an integer satisfying the conclusion of Proposition \ref{prop-contrac-cone} and so that $B_{G,H}(n_0)$ contains $K_1$, and we check that $K = B_{G,H}(n_0)$ satisfies the conclusion. It follows from the last paragraph that the subgroup generated by $K$ together with $t^{n_0}$ is equal to $H \rtimes \left\langle t^{n_0} \right\rangle$ because $K$ contains $K_1$. Besides it is clear that $K^2 \subset B_{G,H}(2 n_0)$, so the inclusion $t^{n_0} \cdot K^2 \cdot t^{-n_0} \subset K$ follows immediately from the conclusion of Proposition \ref{prop-contrac-cone}.
\end{proof}

The following result provides a sufficient condition on a group $G = H \rtimes \left\langle t \right\rangle$ so that the conjugation by the element $t$ induces a compacting automorphism of the group $H$.

\begin{prop} \label{prop-comp-hyp}
Let $G = H \rtimes \left\langle t \right\rangle$ be a locally compact group such that there is some compact symmetric subset $K \subset H$ containing the identity so that:
\begin{enumerate}[label=(\alph*)]
\item $S = K \cup \{ t\}$ generates the group $G$;
\item $t \cdot K^2 \cdot t^{-1} \subset K$.
\end{enumerate}
Then the automorphism of $H$ induced by the conjugation by $t$ is compacting. 
\end{prop}

\begin{proof}
We check that for every $h \in H$, we have $t^n h t^{-n} \in K$ eventually. The hypotheses imply that $H$ is generated by the increasing union of compact sets $t^{-n} \cdot K \cdot t^{n}$, so that every element of $H$ lies inside $t^{-n} \cdot K^{2^k} \cdot t^{n}$ for some integers $n,k \geq 0$. The latter being included in $t^{-n-k} \cdot K \cdot t^{n+k}$ thanks to $(b)$, the proof is complete.
\end{proof}

We are now able to prove the main result of this section.

\begin{proof}[Proof of Theorem \ref{thm-focal}]
Let $G$ be a focal lacunary hyperbolic group. According to Corollary \ref{cor-H-by-Z}, the group $G$ admits a topological semidirect product decomposition of the form $H \rtimes_{\alpha} \mathbb{Z}$ or $H \rtimes_{\alpha(t)} \mathbb{R}$. To conclude we need to prove that the action of $\alpha$ (resp. $\alpha(1)$) on $H$ is compacting. For the sake of simplicity we denote $\alpha(1)$ by $\alpha$ as well. 

We claim that upon changing $\alpha$ into its inverse, there is some positive power of $\alpha$ satisfying the hypotheses of Proposition \ref{prop-comp-hyp}. In the case when $G = H \rtimes_{\alpha} \mathbb{Z}$ this follows directly from Corollary \ref{prop-sbgfi}. When $G = H \rtimes_{\alpha(t)} \mathbb{R}$, the subgroup $H \rtimes_{\alpha(1)} \mathbb{Z}$ is normal and cocompact in $G$, and therefore focal lacunary hyperbolic as well by Lemma \ref{lem-cone-sametype}, so that Corollary \ref{prop-sbgfi} can also be applied.

Consequently Proposition \ref{prop-comp-hyp} implies that some positive power of $\alpha$ is compacting, and it follows that $\alpha$ is compacting as well.
\end{proof}

%\begin{proof}[Proof of Theorem \ref{thm-focal}]
%We prove that any focal lacunary hyperbolic group $G$ has a cocompact subgroup which is hyperbolic. This implies that $G$ itself is hyperbolic, and by Proposition \ref{prop-same-type-general} $G$ must be focal hyperbolic.

%According to Lemma \ref{lem-H-by-Z}, the group $G$ has a cocompact normal subgroup $G'$ of the form $G' = H \rtimes \mathbb{Z}$. Now by Lemma \ref{lem-cone-sametypebis}, the group $G'$ is focal lacunary hyperbolic as well. So we can apply Corollary \ref{prop-sbgfi}, which implies that $G'$ has a finite index subgroup $G''$ satisfying the hypotheses of Proposition \ref{prop-comp-hyp}. Therefore $G''$ must be hyperbolic, and so is its finite index overgroup $G'$, so the proof is complete. 
%\end{proof}

\subsection{Application to locally compact groups with asymptotic cut-points} \label{subsec-cut}

Recall that in a geodesic metric space $X$, a point $x \in X$ is a cut-point if $X \setminus \{x \}$ is not connected. In a sense, the property of having cut-points in asymptotic cones can be seen as a very weak form of hyperbolicity. Examples of finitely generated groups with cut-points in all their asymptotic cones include relatively hyperbolic groups \cite[Theorem 1.11]{DS} or mapping class groups of punctured surfaces \cite[Theorem 7.1]{Beh}. Actually relatively hyperbolic groups and mapping class groups are examples of the so-called acylindrically hyperbolic groups, and it is proved in \cite{Sisto} that any acylindrically hyperbolic group has cut-points in all its asymptotic cones.

Recall that a law is a non-trivial reduced word $w(x_1,\ldots,x_n)$ in the letters $x_1, \ldots, x_n$. A group $G$ is said to satisfy the law $w(x_1,\ldots,x_n)$ if $w(g_1,\ldots,g_n) = 1$ in $G$ for every $g_1, \ldots, g_n \in G$. Examples of groups satisfying a law are solvable groups or groups of finite exponent. In \cite[Theorem 6.12]{DS}, Drutu and Sapir proved that if a finitely generated group $G$ satisfies a law, then $G$ does not have cut-points in any asymptotic cone, unless $G$ is virtually cyclic. However, this result does not hold in the realm of locally compact groups. For example for every local field $\mathbb{K}$, the affine group $\mathbb{K} \rtimes \mathbb{K}^*$ is a non-elementary hyperbolic LC-group and is solvable of class two.

We will extend the result of Drutu and Sapir to locally compact compactly generated groups in Theorem \ref{thm-law-cutpt} below, by proving that if $G$ is a group satisfying a law that is neither an elementary hyperbolic group nor a focal hyperbolic group, then $G$ does not have cut-points in any of its asymptotic cones. Before doing this, let us derive the following consequence of Theorem \ref{thm-focal}.

\begin{prop} \label{prop-lac-law-hyp}
Let $G$ be a locally compact lacunary hyperbolic group. If $G$ satisfies a law then $G$ is hyperbolic.
\end{prop}

\begin{proof}
Let $\cone$ be an asymptotic cone of $G$ that is a real tree. Note that since the group $G$ satisfies a law, the same holds for the group $\pc(G,\mathbf{s})$. Clearly we can assume that $\cone$ is not a point. If $\cone$ is a line, then by Lemma \ref{lem-cone-lineal} the group $G$ is elementary hyperbolic. So we may assume that $\cone$ is not a line, and it follows that the action of $\pc(G,\mathbf{s})$ on $\cone$ is either focal or of general type. But it cannot be of general type, because otherwise this would imply that $\pc(G,\mathbf{s})$ contains a non-abelian free subgroup \cite[Theorem 2.7]{Cul-Mor}, which is a contradiction with the fact that $\pc(G,\mathbf{s})$ satisfies a law. Therefore the action of $\pc(G,\mathbf{s})$ on $\cone$ is focal, and it follows from Theorem \ref{thm-focal} that $G$ is a focal hyperbolic group.
\end{proof}

\begin{rmq} \label{rmq-connected-lac-hyp}
Since the properties of being lacunary hyperbolic and of being hyperbolic are invariant under quasi-isometries, Proposition \ref{prop-lac-law-hyp} still holds for groups quasi-isometric to a group satisfying a law.
\end{rmq}

Although it is not stated explicitly in these terms, the following result can be derived from the work of Drutu and Sapir. For an introduction to the concept of tree-graded spaces, we refer the reader to \cite{DS}.

\begin{prop}[Drutu-Sapir] \label{prop-DS-reform}
Let $G$ be a locally compact compactly generated group satisfying a law. If $\mathcal{C} = \cone$ is an asymptotic cone of $G$ with cut-points, then $\mathcal{C}$ must be a real tree.
\end{prop}

\begin{proof}
Since by assumption $\mathcal{C}$ has cut-points, it follows from Lemma 2.31 of \cite{DS} that $\mathcal{C}$ is tree-graded with respect to a collection of proper subsets. Assume by contradiction that $\mathcal{C}$ is not a real tree. Then we can apply Proposition 6.9 of \cite{DS} to the action of $\pc(G,\mathbf{s})$ on $\mathcal{C}$, and we obtain that $\pc(G,\mathbf{s})$ contains a non-abelian free subgroup. On the other hand since the group $G$ satisfies a law, $\pc(G,\mathbf{s})$ cannot contain a non-abelian free group. Contradiction.
\end{proof}

The following theorem generalizes to the realm of locally compact compactly generated groups the aforementioned result of Drutu and Sapir about finitely generated groups satisfying a law.

\begin{thm} \label{thm-law-cutpt}
Let $G$ be a locally compact compactly generated group satisfying a law. If $G$ has cut-points in one of its asymptotic cones, then $G$ is either an elementary or a focal hyperbolic group.
\end{thm}

\begin{proof}
We let $\mathcal{C}$ be an asymptotic cone of $G$ with cut-points. Since the group $G$ satisfies a law, it follows from Proposition \ref{prop-DS-reform} that $\mathcal{C}$ is a real tree. Therefore $G$ is lacunary hyperbolic, and the conclusion then follows from Proposition \ref{prop-lac-law-hyp}.
\end{proof}

Since the property of having cut-points in one asymptotic cone is a quasi-isometry invariant, the following result follows immediately from the contrapositive of Theorem \ref{thm-law-cutpt}.

\begin{cor} \label{cor-qi-law-cutpt}
Let $G$ be a compactly generated group that is quasi-isometric to a group satisfying a law. If $G$ is not a hyperbolic group then $G$ does not have cut-points in any of its asymptotic cones.
\end{cor}

In particular since connected-by-compact locally compact groups, or compactly generated linear algebraic groups over an ultrametric local field of characteristic zero, are quasi-isometric to a solvable group, we deduce the following result.

\begin{cor} \label{cor-connected-cut}
Let $G$ be a locally compact compactly generated group. Assume that $G$ is either connected-by-compact, or a linear algebraic group over an ultrametric local field of characteristic zero. If $G$ is not a hyperbolic group then $G$ does not have cut-points in any of its asymptotic cones.
\end{cor}

Note that by Corollary $3$ of \cite{CT11}, we have a complete description of connected Lie groups or linear algebraic groups over a non-Archimedean local field of characteristic zero that are non-elementary hyperbolic. For example in the case of a connected Lie group, $G$ is either isomorphic to a semidirect product $N \rtimes (K \times \mathbb{R})$, where $N$ is a simply connected nilpotent Lie group, $K$ is a compact connected Lie group and the action of $\mathbb{R}$ on $N$ is contracting; or the quotient of $G$ by its maximal compact normal subgroup is isomorphic to a rank one simple Lie group with trivial center. So it follows from Corollary \ref{cor-connected-cut} that if a connected Lie group is not of this form, then it does not have cut-points in any of its asymptotic cones. %A very similar statement holds for algebraic groups over a non-Archimedean local field of characteristic zero.

\begin{rmq}
Here is another proof of Corollary \ref{cor-connected-cut} when $G$ is a connected-by-compact locally compact group. Argue by contradiction and assume that $G$ admits one asymptotic cone $\mathcal{C}$ with cut-points. Since $G$ is quasi-isometric to a solvable group, according to Proposition \ref{prop-DS-reform} the asymptotic cone $\mathcal{C}$ must be a real tree. Now since connected-by-compact groups are compactly presented (see for example \cite[Proposition 8.A.16]{Cor-dlH}), the group $G$ must be hyperbolic by Corollary \ref{cor-cp-lac} below. Contradiction.
\end{rmq}

\section{Structural results for locally compact lacunary hyperbolic groups} \label{sec-gen-type}

\subsection{Identity component in lacunary hyperbolic groups} 

Recall that a locally compact compactly generated group $G$ is \textit{lacunary hyperbolic of general type} if it admits one asymptotic cone $\cone$ that is a real tree and such that the action of $\pc(G,\mathbf{s})$ on $\cone$ is of general type. It turns out that, apart from the case of hyperbolic LC-groups, every lacunary hyperbolic group is of general type. This will be proved in Theorem \ref{thm-typegen} below.

It is proved in \cite[Proposition 6.1]{DS} that if a finitely generated group $G$ has one asymptotic cone that is a line, then $G$ is virtually infinite cyclic. The following lemma is an extension of this result to coarsely connected metric groups. In particular it encompasses the case of a closed compactly generated subgroup $H$ of a locally compact compactly generated group $G$, where $H$ is endowed with the induced word metric from $G$.

\begin{lem} \label{lem-cone-lineal}
Let $(\Gamma,d)$ be a group equipped with a coarsely connected left-invariant metric. If $(\Gamma,d)$ admits one asymptotic cone that is quasi-isometric to the real line, then $\Gamma$ admits an infinite cyclic cobounded subgroup.
\end{lem}

\begin{proof}
If $\mathcal{C} = \mathrm{Cone}^{\omega}(\Gamma,d,\mathbf{s})$ is an asymptotic cone of $(\Gamma,d)$ that is quasi-isometric to the real line, the action of $\pc(\Gamma,d,\mathbf{s})$ on $\mathcal{C}$ is lineal. Therefore $\pc(\Gamma,d,\mathbf{s})$ contains some hyperbolic element $\gamma = (\gamma_n)$, and there exists $\ell > 0$ such that the $\ell$-neighbourhood of the $\left\langle \gamma \right\rangle$-orbit of the point $(e)^{\omega}$ is the entire $\mathcal{C}$.

For every $n \geq 1$, we let $\Gamma_n$ be the subgroup of $\Gamma$ generated by $\gamma_n$. We claim that \mbox{$\omega$-almost} surely, $\Gamma$ is contained in the $(\ell+1)s_n$-neighbourhood of $\Gamma_n$. Let us argue by contradiction and assume that \mbox{$\omega$-almost} surely there exists $x_n \in \Gamma$ such that $d(x_n,\Gamma_n) \geq (\ell +1)s_n$. Since $(\Gamma,d)$ is coarsely connected, we can assume that $d(x_n,\Gamma_n) \leq (\ell +1)s_n + c$ for some constant $c>0$. Upon multiplying $x_n$ on the left by an element of $\Gamma_n$, we can moreover assume that $d(x_n,\Gamma_n) = d(x_n,e)$, which implies that the sequence $(x_n)$ defines a point $x \in \mathcal{C}$. But by construction, \mbox{$\omega$-almost} surely $d(x_n,\gamma_n^i) \geq (\ell +1)s_n$ for every $i \in \mathbb{Z}$, so the point $x$ is at distance at least $(\ell +1)$ from any point in the $\left\langle \gamma \right\rangle$-orbit of the point $(e)^{\omega}$. Contradiction.
\end{proof}

\begin{thm} \label{thm-typegen}
Let $G$ be a locally compact lacunary hyperbolic group. Then exactly one of the following holds:
\begin{enumerate}[label=(\alph*)]
\item $G$ is either an elementary or a focal hyperbolic group;
\item for every asymptotic cone $\cone$ that is a real tree, the action of $\pc(G,\mathbf{s})$ on $\cone$ is of general type.
\end{enumerate}
\end{thm}

\begin{proof}
Let $\mathcal{C} = \cone$ be an asymptotic cone of $G$ that is a real tree. By homogeneity $\mathcal{C}$ can be either a point, a line, or such that every point is branching with the same branching cardinality. The case when $\mathcal{C}$ is a point is trivial, as it easily implies that the group $G$ is compact. If $\mathcal{C}$ is a line then $G$ must have an infinite cyclic discrete and cocompact subgroup by Lemma \ref{lem-cone-lineal}. So we may assume that $\mathcal{C}$ is neither a point nor a line. This implies that if the action of $\pc(G,\mathbf{s})$ on $\mathcal{C}$ is not of general type, then it is focal, and by Theorem \ref{thm-focal} this implies that $G$ is focal hyperbolic.
\end{proof}

We now aim to establish some structural results about locally compact lacunary hyperbolic groups. Since any topological group naturally lies into an extension with a connected kernel and a totally disconnected quotient, it is natural to wonder what can be said about the identity component of a locally compact lacunary hyperbolic group. Recall that even for hyperbolic LC-groups, it may happen that the identity component is neither compact nor cocompact. Take for example the semidirect product $(\mathbb{R} \times \mathbb{Q}_p) \rtimes \mathbb{Z}$, where the action of $\mathbb{Z}$ is by multiplication by $1/2$ on $\mathbb{R}$ and by $p$ on $\mathbb{Q}_p$. However, if $G$ is a hyperbolic LC-group of general type, it follows from \cite[Proposition 5.10]{CCMT} that the identity component of $G$ is either compact or cocompact. We will extend this result to lacunary hyperbolic groups in Theorem \ref{thm-composante-typegen} below. 

Recall that if $G$ is a locally compact group, the Braconnier topology is a Hausdorff topology on the group $\mathrm{Aut}(G)$ of topological automorphisms of $G$. For an introduction to this topology, see for example \cite[Appendix I]{CaMo}. 

\begin{lem} \label{lem-grth}
Let $G$ be a $\sigma$-compact locally compact group, and $N \triangleleft G$ a closed normal subgroup with trivial center and finite outer automorphism group. Assume moreover that the group $\mathrm{Inn}(N)$ of inner automorphisms of $N$ is closed in $\mathrm{Aut}(N)$. Then $G$ has a finite index open subgroup that is topologically isomorphic to the direct product of $N$ with its centralizer in $G$. 
\end{lem}

\begin{proof}
If we let $C$ be the centralizer of $N$ in $G$, we want to prove that the subgroup $N C$ is open in $G$, has finite index and is topologically the direct product of $N$ and $C$. Since $N$ is a closed normal subgroup of $G$, the action of $G$ by conjugation on $N$ yields a continuous map $G \rightarrow \mathrm{Aut}(N)$ \cite[Theorem 26.7]{HR}. Being the preimage of the closed finite index subgroup $\mathrm{Inn}(N)$ of $\mathrm{Aut}(N)$ under this map, the subgroup $N C$ is a closed finite index (and hence open) subgroup of $G$. It follows that $NC$ is a $\sigma$-compact locally compact group, and we deduce that the natural epimorphism $N \times C \rightarrow NC$ is a quotient morphism between topological groups. Since it is clearly onto, and injective because $N$ has trivial center, it is an isomorphism of topological groups. 
\end{proof}

\begin{thm} \label{thm-composante-typegen}
Let $G$ be a locally compact lacunary hyperbolic group of general type. Then $G^\circ$ is either compact or cocompact in $G$.
\end{thm}

\begin{proof} %[Proof of Theorem \ref{thm-composante-typegen}]
According to Corollary \ref{cor-red-lie} there exists a compact characteristic subgroup $W$ of $G$ contained in $G^\circ$ such that $G^{\circ} / W$ is a connected Lie group without non-trivial compact normal subgroups. Now by Lemma \ref{lem-cone-sametype}, the group $G/W$ is lacunary hyperbolic of general type as well, so the proof can be reduced to the case when $G^\circ$ is a connected Lie group without non-trivial compact normal subgroups.

Let $\mathcal{C} = \cone$ be an asymptotic cone of $G$ that is a real tree and such that the action of $\pc(G,\mathbf{s})$ on $\mathcal{C}$ is of general type. According to Lemma \ref{lem-normal}, the action of $\pc_G(G^\circ, \mathbf{s})$ on $\mathcal{C}$ is either bounded or of general type. Since $G^\circ$ is compactly generated, if the action of $\pc_G(G^\circ, \mathbf{s})$ on $\mathcal{C}$ is bounded then $G^\circ$ is compact by Lemma \ref{lem-comp-fixpt}. So we may assume that this action is of general type and we will prove that $G^\circ$ is cocompact in $G$.

We denote by $R$ the non-connected solvable radical of $G^\circ$, that is its largest normal solvable subgroup. It is a closed, compactly generated subgroup of $G^\circ$, and being characteristic in the normal subgroup $G^\circ$, the subgroup $R$ is normal in $G$. We will prove that $R$ is reduced to the identity. For the same reason as above, the action of $\pc_G(R, \mathbf{s})$ on $\mathcal{C}$ must be either bounded or of general type. However it cannot be of general type because otherwise $\pc_G(R, \mathbf{s})$ would contain a non-abelian free subgroup (see Theorem 2.7 in \cite{Cul-Mor}), which is clearly impossible because $\pc_G(R,\mathbf{s})$ is a solvable group. Therefore the action of $\pc_G(R, \mathbf{s})$ on $\mathcal{C}$ is bounded, and by Lemma \ref{lem-comp-fixpt} this implies that $R$ is a compact subgroup. But $G^\circ$ is assumed not to contain any non-trivial compact normal subgroup, so $R$ must be trivial.

It follows that $G^\circ$ is a semisimple Lie group with trivial center, and consequently $G^\circ$ has finite outer automorphism group. So we are in position to apply Lemma \ref{lem-grth}, and we obtain that $G$ admits a finite index open subgroup decomposing as a topological direct product $G' = G^\circ \times Q$. Now since $G'$ has finite index in $G$, $\mathrm{Cone}^{\omega}(G',\mathbf{s}) \simeq \mathrm{Cone}^{\omega}(G^\circ,\mathbf{s}) \times \mathrm{Cone}^{\omega}(Q,\mathbf{s}) $ is a real tree. This implies that either $\mathrm{Cone}^{\omega}(G^\circ,\mathbf{s})$ or $\mathrm{Cone}^{\omega}(Q,\mathbf{s})$ is a point, that is either $G^\circ$ or $Q$ is compact. But by assumption $G^\circ$ is not compact so $Q$ must be compact, and the conclusion follows.
\end{proof}

As a consequence of this result, we deduce the following property for locally compact lacunary hyperbolic groups.

\begin{prop} \label{prop-hyp-co}
If $G$ is a locally compact lacunary hyperbolic group, then either $G$ is hyperbolic or $G$ has a compact open subgroup.
\end{prop}

\begin{proof}
According to Lemma \ref{lem-cone-lineal}, if $G$ is not an elementary hyperbolic LC-group, then $G$ must be either focal lacunary hyperbolic or lacunary hyperbolic of general type. If $G$ is focal lacunary hyperbolic then $G$ is focal hyperbolic by Theorem \ref{thm-focal}. Now if $G$ is lacunary hyperbolic of general type, then according to Theorem \ref{thm-composante-typegen} the identity component $G^\circ$ is either compact or cocompact in $G$. In the latter case $G$ must be hyperbolic (see Remark \ref{rmq-connected-lac-hyp}), and in the former $G$ has a compact open subgroup by van Dantzig's theorem.
\end{proof}

\subsection{Characterization of lacunary hyperbolic groups}

\subsubsection{Cartan-Hadamard Theorem}

This paragraph consists of a recall of a Cartan-Hadamard type theorem due to Gromov, and its application to lacunary hyperbolic groups due to Kapovich and Kleiner, stated for topological groups rather than discrete ones. 

Let $(X,d)$ be a non-empty geodesic metric space, $x_0 \in X$ a base point and $c>0$. A $c$-loop based at $x_0$ is a sequence of points $x_0 = x_1,x_2,\ldots,x_n=x_0$ such that $d(x_i,x_{i+1}) \leq c$ for every $i=1,\ldots,n-1$. Two $c$-loops are said to be $c$-elementarily homotopic if one of them can be obtained from the other by inserting a new point, and $c$-homotopic if they are the extremities of a finite sequence of $c$-loops such that any two consecutive terms are $c$-elementarily homotopic. Recall that $X$ is $c$-large scale simply connected if any $c$-loop based at $x_0$ is $c$-homotopic to the trivial loop.

The following result can be deduced from \cite[Part III.H Lemma 2.6]{BH}.

\begin{prop} \label{prop-lssc-BH}
There exists some universal constant $C>0$ so that every geodesic $\delta$-hyperbolic metric space is $C\delta$-large scale simply connected.
\end{prop}

The following result appears as a large scale analogue of Cartan-Hadamard Theorem in metric geometry. The idea of this local-global principle goes back to \cite{GroHyp}, but the version we use here is inspired from Theorem 8.3 of the Appendix of \cite{OOS} (see also Chapter 8 of \cite{Bow}).

\begin{thm} \label{CarHad}
There exist some constants $c_1,c_2,c_3 > 0$ such that the following holds: every geodesic, $c$-large scale simply connected metric space $X$ with the property that there exists some $R \geq c_1 c$ such that every ball in $X$ of radius $R$ is $c_2 R$-hyperbolic; is $c_3 R$-hyperbolic.
\end{thm}

Kapovich and Kleiner \cite[Appendix]{OOS} observed that geodesic metric spaces with one asymptotic cone that is a real tree fulfill the assumption of local hyperbolicity appearing in Theorem \ref{CarHad}, which yields the following corollary.

\begin{cor} \label{lssc}
Let $(X,d$) be a homogeneous, geodesic, $c$-large-scale simply connected metric space. If $X$ is lacunary hyperbolic then $X$ is hyperbolic.
\end{cor}

\begin{proof}
Let $e \in X$ be a base point, $\mathbf{s}$ a scaling sequence and $\omega$ a non-principal ultrafilter such that $\mathrm{Cone}^{\omega}(X,d,\mathbf{s})$ is a real tree. Then \mbox{$\omega$-almost} surely, the ball or radius $s_n$ in $X$ around $e$ is $\delta_n$-hyperbolic, with $\delta_n = o(s_n)$. But since $X$ is homogeneous, every ball in $X$ or radius $s_n$ is $\delta_n$-hyperbolic. Now for some large enough $n$ we have $s_n \geq c_1 c$ and $\delta_n/s_n \leq c_2$, so it follows from Theorem \ref{CarHad} that $X$ is hyperbolic.
\end{proof}

Since for a locally compact compactly generated group, compact presentability can be characterized in terms of large scale simple connectedness (see for example \cite[Proposition 8.A.3]{Cor-dlH}), we obtain the following result, which is the topological counterpart of \cite[Theorem 8.1]{OOS} by Kapovich and Kleiner. 

\begin{cor} \label{cor-cp-lac}
Any compactly presented group that is lacunary hyperbolic is a hyperbolic group.
\end{cor}

%\begin{proof}
%If $G$ is a compactly generated locally compact group, then the inclusion of $G$ into any of its Rips complexes is a quasi-isometry. As being lacunary hyperbolic is invariant under quasi-isometry, all of these Rips complexes are also lacunary hyperbolic. Now since $G$ is compactly presented, there exists $c>0$ such that $\mathrm{Rips}_c^2(G)$ is simply connected (see \cite{Cor-dlH} for example). By Corollary \ref{lssc} $\mathrm{Rips}_c^2(G)$ is hyperbolic, and so is $G$.
%\end{proof}

%\begin{cor} \label{cor-connected-lac}
%Connected-by-compact lacunary hyperbolic groups are then $G$ is a hyperbolic group.
%\end{cor}

%\begin{proof}
%Connected-by-compact compactly generated groups are compactly presented (see for example \cite[Proposition 8.A.16]{Cor-dlH}), so the conclusion follows from Corollary \ref{cor-cp-lac}.
%\end{proof}

\subsubsection{Characterization of locally compact lacunary hyperbolic groups}

We are now able to generalize to the locally compact setting the structural theorem of Olshanskii, Osin, Sapir \cite{OOS} for finitely generated lacunary hyperbolic groups.

\begin{thm} \label{thm-struct-lac}
Let $G$ be a compactly generated locally compact group with a compact open subgroup. Then the following assertions are equivalent:
\begin{enumerate}[label=(\roman*)]
\item $G$ is lacunary hyperbolic;
\item There exists a scaling sequence $\mathbf{s}$ such that for every non-principal ultrafilter $\omega$, the asymptotic cone $\cone$ is a real tree;
\item There exists a hyperbolic LC-group $G_0$ acting on a locally finite tree, transitively and with compact open stabilizers on the set of vertices, and an increasing sequence of discrete normal subgroups $N_n$, whose discrete union $N$ is such that $G$ is topologically isomorphic to $G_0 / N$; and if $S$ is a compact generating set of $G_0$ and \[\rho_n = \min \{\left|g\right|_S \, : \, g \in N_{n+1} \! \setminus \! N_n \},\] then $G_0 / N_n$ is $\delta_n$-hyperbolic with $\delta_n =o(\rho_n)$.
\end{enumerate}
\end{thm}

The proof of the implication $(iii) \Rightarrow (ii)$ is similar to the one for discrete groups, so we choose not to repeat it here and refer the reader to \cite[p.16]{OOS}. The implication $(ii) \Rightarrow (i)$ being trivial, we only have to prove $(i) \Rightarrow (iii)$. %Compared to the case of finitely generated groups, the idea is to replace Cayley graphs by Cayley-Abels graphs and use the construction of Proposition \ref{prop-free}.

%\begin{proof}
%\begin{itemize}
%\item $(iii) \Rightarrow (ii)$. Let $(s_n)$ be a sequence of positive numbers such that $\delta_n <\!\!< s_n <\!\!< \rho_n$ (take for example the geometric mean $s_n = \sqrt{\delta_n \rho_n}$), and let $\omega$ be a non-principal ultrafilter. Note that the surjection $G_n \twoheadrightarrow G$ is injective on the ball or radius $\rho_n$ **FAUX**. Since $s_n <\!\!< \rho_n$, it follows that \mbox{$\omega$-almost} surely the ball of radius $s_n$ in $G$ is isometric to the ball of radius $s_n$ in $G_n$, and is therefore $\delta_n$-hyperbolic. Now the conclusion is a straightforward consequence of Lemma \ref{cone-arbre} since $\delta_n <\!\!< s_n$.
%**UTILISER DRUTU-SAPIR 3.34 EN SUPPOSANT QU'IL Y A UN TRIANGLE GEODESIQUE SIMPLE DANS CONE ?**
%\end{itemize}
%\end{proof}

%\begin{prop}
%Let $G$ be a compactly generated locally compact group with a compact open subgroup. If $G$ is lacunary hyperbolic then there exists a %hyperbolic group $G_0$ (acting geometrically on a locally finite tree) and an increasing sequence of discrete normal subgroups $(N_n)$ of %union $N$ such that $G$ is topologically isomorphic to $G_0 / N$, and if $S$ is a compact generating set of $G_0$ and \[ \rho_n = \min %\left\{\left|g\right|_S \, : \, g \in N_{n+1} \setminus N_n \right\},\] then $G_n := G_0 / N_n$ is $\delta_n$-hyperbolic with $\delta_n %=o(\rho_n)$.
%\end{prop}

\begin{proof}[Proof of $(i) \Rightarrow (iii)$]
Let $G$ be a lacunary hyperbolic group with a compact open subgroup. We let $G_0$ and $\pi: G_0 \rightarrow G$ be as in Proposition \ref{prop-free}. Recall that $G_0$ is a locally compact compactly generated group acting geometrically on a locally finite tree and $\pi$ is an open morphism from $G_0$ onto $G$ with discrete kernel $N$. Let $\omega$ be a non-principal ultrafilter and $\mathbf{s}$ a scaling sequence such that $\cone$ is a real tree. Choose a compact open subgroup $K$ of $G_0$ intersecting $N$ trivially, and a $K$-bi-invariant compact generating set $S$ of $G_0$. For every $k \geq 1$, let $N_k$ be the normal subgroup of $G_0$ generated by elements of $N$ of word length at most $d_k$ with respect to $S$, and set $G_k = G_0 / N_k$. Note that since $\mathbf{s}$ is an increasing sequence tending to infinity, by construction $(N_k)$ is an increasing sequence of normal subgroups of $G_0$ whose union is $N$. This can be rephrased by saying that we have an infinite sequence of locally compact groups and quotient morphisms \[ G_0 \twoheadrightarrow \cdots \twoheadrightarrow G_k \twoheadrightarrow G_{k+1} \twoheadrightarrow \cdots \] whose direct limit is topologically isomorphic to the group $G$. Observe that the injectivity radius of the map $G_k \twoheadrightarrow G$ is larger than $d_k$, and a fortiori the same holds for the injectivity radius of the map $G_k \twoheadrightarrow G_{k+1}$.

For every $k \geq 1$, we push the pair $(K,S)$ in $G_k$ and in $G$, and we denote by $X_k$ (resp. $X$) the Cayley-Abels graph of $G_k$ (resp. $G$) with respect to this compact open subgroup and compact generating set. By abuse of notation, we still denote by $K$ the image of the subgroup $K$ in $G_k$. To the above sequence of groups and epimorphisms corresponds an infinite sequence of coverings of graphs \[ X_0 \twoheadrightarrow \cdots \twoheadrightarrow X_k \twoheadrightarrow X_{k+1} \twoheadrightarrow \cdots \] Note that the map $X_k \twoheadrightarrow X$ is injective on the ball $B_{X_k}(K,d_k)$ of radius $d_k$ around the vertex $K$.

Now since $G$ is quasi-isometric to its Cayley-Abels graph $X$, their asymptotic cones $\mathrm{Cone}^{\omega}(X,\mathbf{s})$ and $\cone$ are bi-Lipschitz homeomorphic. It follows that $\mathrm{Cone}^{\omega}(X,\mathbf{s})$ is a real tree, and therefore \mbox{$\omega$-almost} surely the ball of radius $d_k$ in $X$ is $\delta_k$-hyperbolic with $\delta_k = o(d_k)$. By the above observation on the injectivity radius of the map $X_k \twoheadrightarrow X$, the same is true in $X_k$. According to Proposition \ref{prop-lssc-BH}, the ball $B_{X_k}(K,d_k)$ is $O(\delta_k)$-large scale simply connected. But by construction of the group $G_k$, any loop in $X_k$ is built from loops of length at most $d_k$, so it follows that the entire graph $X_k$ is $C_k$-large scale simply connected, with $C_k = O(\delta_k)$.

Now let us pick a sequence $(\Delta_k)$ such that $\delta_k <\!< \Delta_k <\!< d_k$. If we let $c_1,c_2,c_3$ be the constants from Theorem \ref{CarHad}, then \mbox{$\omega$-almost} surely $\Delta_k \geq c_1 C_k$ and $\Delta_k \geq \delta_k / c_2$. So we are in position to apply Theorem \ref{CarHad}, which implies that \mbox{$\omega$-almost} surely $X_k$ is $c_3 \Delta_k$-hyperbolic.

Now as observed earlier, the injectivity radius $\rho_k$ of $G_k \twoheadrightarrow G_{k+1}$ satisfies $\rho_k \geq d_k$. Since $\Delta_k = o(d_k)$, we clearly have $\Delta_k = o(\rho_k)$. It follows that $\omega$ almost surely, the graph $X_k$ (and a fortiori the group $G_k$) is $o(\rho_k)$-hyperbolic, and the conclusion follows.
\end{proof}

The next proposition establishes some stability properties of the class of locally compact lacunary hyperbolic groups. We note that, as observed in \cite{OOS}, the class of finitely generated lacunary hyperbolic groups is not stable under free product.

\begin{prop} \label{prop-lac-stab}
The class of locally compact lacunary hyperbolic groups is stable under taking:
\begin{enumerate}[label=(\alph*)]
\item a semidirect product with a compact group;
\item an HNN-extension over some compact open subgroup;
\item an amalgamated product with a hyperbolic LC-group over some compact open subgroup.
\end{enumerate}
\end{prop}

\begin{proof}
The statement $(a)$ is trivial. Let us prove $(b)$. Let $G$ be a lacunary hyperbolic group, $K,L$ two compact open subgroups, $\varphi: K \rightarrow L$ a topological isomorphism, and $G' = \mathrm{HNN}(G,K,L,\varphi)$ the corresponding HNN-extension. We want to prove that $G'$ is lacunary hyperbolic. If $G$ is hyperbolic then there is nothing to prove because since $K,L$ are compact, the group $G'$ is hyperbolic as well. Otherwise $G$ has a compact open subgroup by Proposition \ref{prop-hyp-co}, and we let $G_0$ be a hyperbolic LC-group and $(N_n)$ an increasing sequence of discrete normal subgroups as in Theorem \ref{thm-struct-lac}. There exists an integer $n_0 \geq 1$ such that for every $n \geq n_0$, the group $G_n$ has subgroups isomorphic to $K$ and $L$, which we still denote by $K$ and $L$ by abuse of notation. Let us form the HNN-extension $G_n' = \mathrm{HNN}(G_n,K,L,\varphi)$. Since $G_n$ is $\delta_n$-hyperbolic and $K,L$ are compact, the group $G_n'$ is $\delta_n'$-hyperbolic. Moreover since $K,L$ have bounded diameter in $G_n$, we have $\delta_n' = O(\delta_n)$. Now the epimorphism $\alpha_n: G_n \twoheadrightarrow G_{n+1}$ naturally extends to $\alpha_n': G_n' \twoheadrightarrow G_{n+1}'$ by mapping the stable letter to itself, and the injectivity radius $\rho_n'$ of $\alpha_n'$ is equal to the injectivity radius $\rho_n$ of $\alpha_n$. Since by assumption $\rho_n < \! < \delta_n$, we have $\rho_n' < \! < \delta_n'$, and the fact that $G'$ is lacunary hyperbolic follows from the implication $(iii) \Rightarrow (i)$ in Theorem \ref{thm-struct-lac}.

The case $(c)$ of an amalgamated product with a hyperbolic LC-group over some compact open subgroup is analogous, and relies on the fact that the amalgamated product of two hyperbolic LC-groups over some compact open subgroup remains hyperbolic, with a control on the hyperbolicity constant in terms of the hyperbolicity constants of the two groups and the diameter of the compact subgroup.
\end{proof}

\begin{ex}
Here is a construction providing examples of locally compact lacunary hyperbolic groups with a non-discrete topology. Let $\Gamma$ be a discrete lacunary hyperbolic group, and $G$ a hyperbolic LC-group with some compact open subgroup $U$. Let us consider the semidirect product $H = (\ast_{G/U} \Gamma) \rtimes G$, where $G$ acts on the free product $\ast_{G/U} \Gamma$ by permuting the factors according to the natural action of $G$ on $G/U$, and the topology on $H$ is such that the subgroup $G$ is open. Equivalently, $H$ can be defined as the topological amalgamated product of $\Gamma \times U$ with $G$ over the subgroup $U$. It follows from the statements $(a)$ and $(c)$ of Proposition \ref{prop-lac-stab} that the group $H$ is lacunary hyperbolic. Note that the group $H$ may be far from discrete, because for example $H$ is non-unimodular as soon as $G$ is.
\end{ex} 

\section{Subgroups of lacunary hyperbolic groups} \label{sec-subgroups}

In this section we carry on the investigation started in \cite{OOS} of groups that may appear as subgroups of lacunary hyperbolic groups. 

\subsection{Quasi-isometrically embedded normal subgroups}

It is a classical result that if $G$ is a hyperbolic LC-group, and $N$ a compactly generated quasi-isometrically embedded normal subgroup of $G$, then $N$ must be either compact or cocompact in $G$. The following proposition, which is new even for discrete groups, is a generalization of this result to the realm of lacunary hyperbolic groups.

\begin{prop} \label{prop-normal-compact-cocompact}
Let $G$ be a locally compact lacunary hyperbolic group, and $N$ a closed normal subgroup of $G$. Assume that $N$ is compactly generated and quasi-isometrically embedded in $G$. Then $N$ is either compact or cocompact in $G$.
\end{prop}

\begin{proof}
We let $\mathcal{C} = \cone$ be an asymptotic cone of $G$ that is a real tree, and we denote by $\mathcal{C}_N$ the $\pc_G(N,\mathbf{s})$-orbit of $(e)^{\omega} \in \cone$. Since $N$ is compactly generated and quasi-isometrically embedded in $G$, the subset $\mathcal{C}_N$ is a subtree of $\mathcal{C}$ that is clearly invariant by $\pc_G(N,\mathbf{s})$.

First assume that $\pc_G(N,\mathbf{s})$ acts on $\mathcal{C}$ with some hyperbolic element. Then we are in position to apply Lemma \ref{lem-norm-hyp}, which implies that $\mathcal{C}_N$ must be the entire $\mathcal{C}$. The fact that $N$ is cocompact in $G$ then follows from Corollary \ref{lem-cocomp-cone}.

We now have to deal with the case when $\pc_G(N,\mathbf{s})$ does not have any hyperbolic element. We claim that the action of $\pc_G(N,\mathbf{s})$ on $\mathcal{C}$ cannot be horocyclic. Indeed otherwise the action of $\pc_G(N,\mathbf{s})$ on the subtree $\mathcal{C}_N$ would be horocylic as well, which is impossible since a transitive isometric action on a real tree cannot be horocyclic. This implies that if $\pc_G(N,\mathbf{s})$ does not contain any hyperbolic element then $\pc_G(N,\mathbf{s})$ must have a fixed point, and by Lemma \ref{lem-comp-fixpt} this forces the subgroup $N$ to be compact.
\end{proof}

\subsection{Subgroups satisfying a law}

The goal of this paragraph is to exhibit some obstruction for a given group to be a subgroup of a lacunary hyperbolic group.

Recall that if $G$ is a compactly generated group endowed with a compact generating set $S$, and if $H$ is a subgroup of $G$, we denote by $B_{G,H}(n)$ the intersection between $H$ and the ball in $G$ of radius $n \geq 1$ around the identity. If $\mu$ is a left-invariant Haar measure on $G$, a measurable subgroup $H$ is said to have relative exponential growth in $G$ if there exists $\rho > 1$ such that $\rho^n \leq \mu \left( B_{G,H}(e,n) \right)$ for every $n \geq 1$. Note that this condition implies that the subgroup $H$ has positive Haar measure, and hence is open in $G$. If $H_1$ is an open subgroup of $G$, the restriction to $H_1$ of a Haar measure on $G$ is a Haar measure on $H_1$, so if $H_2$ is a subgroup of $H_1$ of relative exponential growth in $H_1$, then $H_2$ has relative exponential growth in $G$. For example a compactly generated open subgroup of exponential growth has relative exponential growth in the ambient group.

\begin{prop} \label{lem-exp-growth-horo}
Let $G$ be a unimodular lacunary hyperbolic group, and $H \leq G$ a subgroup of relative exponential growth in $G$. If $\mathcal{C} = \cone$ is an asymptotic cone of $G$ that is a real tree, then the action of $\pc_G(H,\mathbf{s})$ on $\mathcal{C}$ cannot have a fixed point or be horocyclic.
\end{prop}

\begin{proof}
We shall prove that the action of $\pc_G(H,\mathbf{s})$ on $\mathcal{C}$ cannot be horocyclic. The case of an action with a fixed point can be ruled out with the same kind of arguments, and is actually easier.

Let $S$ be a compact generating set of $G$, and $\mu$ a left-invariant Haar-measure on $G$. We argue by contradiction and assume that the action of $\pc_G(H,\mathbf{s})$ on $\mathcal{C}$ is horocyclic, and denote by $\xi: [0, + \infty [ \rightarrow \mathcal{C}$ the ray emanating from $(e)^{\omega}$ representing the end of $\mathcal{C}$ that is fixed by $\pc_G(H,\mathbf{s})$. Then every element $(h_n) \in \pc_G(H,\mathbf{s})$ such that $|h_n|_S \leq s_n$ fixes the point $\xi(1/2) = (\xi_n)^{\omega}$, that is \[ \mathrm{lim}^{\omega} \frac{d_S(h_n \xi_n, \xi_n)}{s_n} = \mathrm{lim}^{\omega} \frac{|\xi_n^{-1} h_n \xi_n|_S}{s_n} = 0. \] This means that for every $\varepsilon > 0$, \mbox{$\omega$-almost} surely the element $\xi_n^{-1} h_n \xi_n$ has length at most $\varepsilon s_n$, which is equivalent to saying that $h_n$ belongs to $\xi_n \cdot B_{G}(e, \varepsilon s_n) \cdot \xi_n^{-1}$. So for every $\varepsilon > 0$, \mbox{$\omega$-almost} surely \[ B_{G,H}(e,s_n) \subset \xi_n \cdot B_G(e,\varepsilon s_n) \cdot \xi_n^{-1}.\]  Combined with the fact that $G$ is unimodular, we obtain that \mbox{$\omega$-almost} surely \[ \mu \left(B_{G,H}(e,s_n)\right) \leq \mu \left(\xi_n \cdot B_G(e,\varepsilon s_n) \cdot \xi_n^{-1}\right) = \mu \left(B_G(e,\varepsilon s_n)\right) \leq \alpha^{\varepsilon s_n}\] for some constant $\alpha \geq 1$. This implies that \[ \liminf_{n \rightarrow \infty} \frac{\log \mu \left(B_{G,H}(e,s_n)\right)}{s_n} = 0, \] which is a contradiction with the fact that $H$ has relative exponential growth in $G$.
\end{proof}

Let us derive the following consequence of Proposition \ref{lem-exp-growth-horo}, which recovers Theorem 3.18 (c) of \cite{OOS}, and generalizes it to the setting of unimodular locally compact lacunary hyperbolic groups. 

\begin{cor} \label{prop-exponent-growth}
Let $G$ be a unimodular lacunary hyperbolic group, and $H \leq G$ a subgroup of finite exponent. Then $H$ cannot have relative exponential growth in $G$.
\end{cor}

\begin{proof}
For any scaling sequence $\mathbf{s}$, the group $\pc_G(H,\mathbf{s})$ has finite exponent as well. It follows that for any asymptotic cone $\cone$ that is a real tree, the action of $\pc_G(H,\mathbf{s})$ on $\cone$ must have a fixed point or be horocyclic, and $H$ cannot have relative exponential growth in $G$ according to Proposition \ref{lem-exp-growth-horo}.
\end{proof}

We point out that both Corollary \ref{prop-exponent-growth} and Proposition \ref{lem-exp-growth-horo} fail without the assumption that the group is unimodular. Actually the corresponding statements at the level of groups rather than asymptotic cones already fail for hyperbolic LC-groups of general type. Take for example the amalgamated product of $\mathbb{Z}/2\mathbb{Z} \times \mathbb{F}_p[\![t]\!]$ and $\mathbb{F}_p( \! (t) \! ) \rtimes_t \mathbb{Z}$ over the compact open subgroup $\mathbb{F}_p[\![t]\!]$. The resulting group is hyperbolic of general type and non-unimodular. Having relative exponential growth in the open subgroup $\mathbb{F}_p( \! (t) \! ) \rtimes \mathbb{Z}$, the finite exponent subgroup $\mathbb{F}_p( \! (t) \! )$ has relative exponential growth in the ambient group. To see why the conclusion of Proposition \ref{lem-exp-growth-horo} fails, note that the action of $\mathbb{F}_p( \! (t) \! )$ on the quasi-isometrically embedded subgroup $\mathbb{F}_p( \! (t) \! ) \rtimes \mathbb{Z}$ is horocylic, so its action on the entire group must be horocyclic as well. 

\begin{prop} \label{prop-exp-growth-focal}
Let $G$ be a unimodular lacunary hyperbolic group. If $H$ is a subgroup of relative exponential growth in $G$, and if $\mathcal{C} = \cone$ is an asymptotic cone of $G$ that is a real tree, then the action of $\pc_G(H,\mathbf{s})$ on $\mathcal{C}$ cannot be focal.
\end{prop}

\begin{proof}
The argument will be a slight modification of the beginning of the proof of Lemma \ref{lem-focal-unim}. Assume that $\xi: [0, + \infty [ \rightarrow \mathcal{C}$ is a geodesic ray starting at $(e)^{\omega}$ representing and end of $\mathcal{C}$ that is fixed by $\pc_G(H,\mathbf{s})$. Let us fix some $k \geq 1$, and consider $k+1$ points $\xi(1) = x^{(0)},x^{(1)},\ldots,x^{(k)} = \xi(2)$  dividing the interval $[\xi(1),\xi(2)]$ into $k$ segments of equal length. For every $(h_n) \in \pc_G(H,\mathbf{s})$ such that $|h_n|_S \leq s_n$ for every $n \geq 1$, upon changing $(h_n)$ in $(h_n)^{-1}$, there exists some point $x^{(i)}$ such that the distance in $\mathcal{C}$ between $(h_n)\cdot \xi(1)$ and $x^{(i)}$ is at most $1/2k$. This implies that \mbox{$\omega$-almost} surely, the distance in $G$ between $h_n \xi_n(1)$ and $x^{(i)}_n$ is at most $s_n / k$. 

So for every $k \geq 1$, \mbox{$\omega$-almost} surely \[ B_{G,H}(e,s_n) \subset \bigcup_{i=0}^k \left( x_n^{(i)} \cdot B_{G}(e,s_n/k) \cdot \xi_n(1)^{-1} \right)^{\pm 1}, \]
and certainly \[ \begin{aligned} \mu \left(B_{G,H}(e,s_n) \right) & \leq \sum_{i=0}^k 2 \mu \left( x_n^{(i)} \cdot B_{G}(e,s_n/k) \cdot \xi_n(1)^{-1} \right) \\ & = 2 (k+1) \mu \left(B_{G}(e,s_n/k)\right) \\ & \leq 2 (k+1) \alpha^{s_n/k} \end{aligned} \] for some constant $\alpha \geq 1$. Now since $H$ has relative exponential growth in $G$, we obtain that there exists $\rho > 1$ such that for every $k \geq 1$, \mbox{$\omega$-almost} surely $\rho^{s_n} \leq 2 (k+1) \alpha^{s_n/k}$. This implies that $\rho \leq \alpha^{1/k}$ for every $k \geq 1$, which contradicts the fact that $\rho > 1$.
\end{proof}

\begin{prop} \label{prop-law-cyclic}
Let $G$ be a unimodular lacunary hyperbolic group, and $H \leq G$ a compactly generated subgroup of relative exponential growth in $G$. Assume that $H$ does not have a cyclic cocompact subgroup. Then for any asymptotic cone $\mathcal{C} = \cone$ of $G$ that is a real tree, the action of $\pc_G(H,\mathbf{s})$ on $\mathcal{C}$ is of general type.
\end{prop}

\begin{proof}
We carry out a case-by-case analysis of the possible type of the action of $\pc_G(H,\mathbf{s})$ on $\mathcal{C}$, and prove that other types of actions all lead to a contradiction. 

If $\pc_G(H,\mathbf{s})$ fixes a point in $\mathcal{C}$ then Lemma \ref{lem-comp-fixpt} implies that $H$ is compact, which is a contradiction with the fact that $H$ has relative exponential growth. Now assume that the action of $\pc_G(H,\mathbf{s})$ on $\mathcal{C}$ is lineal. Since $H$ is compactly generated, the metric space $(H,d_G)$ is coarsely connected \cite[Proposition 4.B.8]{Cor-dlH}. So we are in position to apply Lemma \ref{lem-cone-lineal} to obtain that $H$ admits an infinite cyclic cocompact subgroup, which is again a contradiction. Finally, it follows from Proposition \ref{lem-exp-growth-horo} that the action of $\pc_G(H,\mathbf{s})$ on $\mathcal{C}$ cannot be horocylic, and according to Proposition \ref{prop-exp-growth-focal} it cannot be focal either.
\end{proof}

We immediately deduce the following result.

\begin{cor} \label{cor-growth-law}
Let $G$ be a unimodular lacunary hyperbolic group. If $H \leq G$ is a compactly generated subgroup of relative exponential growth in $G$ not having $\mathbb{Z}$ as a discrete cocompact subgroup, then $H$ cannot satisfy a law.
\end{cor}

When specified to finitely generated groups, Corollary \ref{cor-growth-law} answers Question 7.2 in \cite{OOS}. As an example, we deduce the following result.

\begin{cor} \label{cor-solv-lac}
Any finitely generated solvable subgroup of a finitely generated lacunary hyperbolic group is virtually cyclic.
\end{cor}

\begin{proof}
Let $H$ be a finitely generated solvable group that is a subgroup of a finitely generated lacunary hyperbolic group $G$. Assume that $H$ has exponential growth. Then $H$ has relative exponential growth in $G$, and according to Corollary \ref{cor-growth-law} the group $H$ must be virtually cyclic, contradiction. Therefore the solvable group $H$ does not have exponential growth, and we deduce that $H$ must be virtually nilpotent \cite{Milnor, Wolf}. In particular $H$ is finitely presented and therefore must be a subgroup of a hyperbolic group \cite[Theorem 3.18 (a)]{OOS}, and the conclusion follows from the fact that any finitely generated virtually nilpotent subgroup of a hyperbolic group is virtually cyclic.
\end{proof}

\nocite{*}
\bibliographystyle{amsalpha}
\bibliography{lacunary}

\end{document}